\documentclass[final,reqno]{siamltex}
\usepackage{amsmath}
\usepackage[table]{xcolor}
\usepackage{subfig}
\usepackage{graphicx}
\usepackage{amsfonts}
\usepackage{tikz}
\usepackage[table]{xcolor}
\usepackage{booktabs}
\usepackage{epsfig,epsf,psfrag}
\usepackage[left=1in, right=1in, top=1in, bottom=1in]{geometry}
\newtheorem{remark}{Remark}
\def\x{{\bf x}}
\def\y{{\bf y}}
\def\s{{\bf s}}
\def\t{{\bf t}}
\def\c{{\bf c}}
\def\j{{\bf j}}

\def\k{{\bf k}}

\def\beq{\begin{equation}}
\def\eeq{\end{equation}}

\def\newpage{\vfill\eject}

\def\bR{{\bf R}}

\def\rprime{\hbox{\hskip.25em\raise.5ex\hbox{$'$}\hskip.15em}}

\def\ddn1{{\frac{\partial}{\partial \nu_{\yb}}}}

\newcommand{\bbR}{{\mathbb R}}

\title{An adaptive fast Gauss transform in two dimensions}
\author{Jun Wang%
\thanks{Courant Institute of Mathematical Sciences,
New York University, New York, New York 10012.
(Present address: 
Flatiron Institute, Simons Foundation, New York, New York 10010.
Email: junwang@flatironinstitute.org.)}
\and
Leslie Greengard%
\thanks{Courant Institute of Mathematical Sciences,
New York University, New York, New York 10012 and
Flatiron Institute, Simons Foundation, New York, New York 10010.
This work was supported in part by the Applied Mathematical
Sciences Program of the U.S. Department of Energy under Contract
DEFGO288ER25053 and by the RiskEcon Lab for Decision Metrics,
Courant Institute.
Email: greengard@cims.nyu.edu.}
}

\begin{document}

\maketitle
\pagestyle{myheadings}
\markboth{\sc J. Wang and L. Greengard}
{\sc An adaptive fast Gauss transform in two dimensions}

\begin{abstract}
A variety of problems in computational physics and engineering
require the convolution of the heat kernel (a Gaussian) with either 
discrete sources, densities supported on boundaries, or continuous volume
distributions. We present a unified fast Gauss transform for this purpose 
in two dimensions, making use of  
an adaptive quad-tree discretization on a unit
square which is assumed to contain all sources. Our implementation permits
either free-space or periodic boundary conditions to be imposed, 
and is efficient for any choice of variance in the Gaussian.
\end{abstract}

\begin{keywords}
fast Gauss transform, heat equation, adaptive mesh refinement
\end{keywords}

\vspace{.1in}

\begin{AMS}
31A10 35K10 65R10 65Y20
\end{AMS}

\section{Introduction}

A variety of problems in applied physics and engineering involve the 
solution of the heat equation
\begin{equation} \label{heateqfree}
\begin{split}
u_{t}(\x,t) &= \Delta u(\x,t) + F(\x,t) \\
u(\x,0) &= f(\x)
\end{split}
\end{equation}
for $t>0$,
in a interior or exterior domain $\Omega$, 
subject to suitable conditions on 
its boundary $\Gamma = \partial\Omega$. For simplicity, we will assume that
these take the form of either Neumann conditions:
\begin{equation}
  \frac{\partial u}{\partial n} (\x,t) = g(\x,t)\ \  {\rm for} \quad \x \in \Gamma,
\label{neumannbc}
\end{equation}
Dirichlet conditions:
\begin{equation}
  u(\x,t) = h(\x,t)\ \ {\rm for} \quad \x \in \Gamma,
\label{dirichletbc}
\end{equation}
or periodic boundary conditions with $\Omega$ the unit square.

In the absence of physical boundaries,  the equations \eqref{heateqfree} 
are well-posed in free space (under mild conditions on the behavior of 
$u$, $f$ and $F$ at infinity) without
auxiliary conditions. Moreover, assuming $F(\x,t)$ and $f(\x)$ are compactly 
supported in the region $\Omega$,
the solution to \eqref{heateqfree} can be expressed at the next time step,
$t = \Delta t$ in closed form as
\begin{equation}
u(\x,\Delta t) = J[f](\x,\Delta t) + V[F](\x,\Delta t)
\label{fullsol}
\end{equation}
with 
\begin{align}
J[f](\x,\Delta t) &= \int_{\Omega} G(\x-\y,\Delta t) f(\y) \, d\y \\
V[F](\x,\Delta t) &= \int_0^{\Delta t} \int_{\Omega} G(\x-\y,\Delta t-\tau) F(\y,\tau) \, d\y d\tau \, .
\label{potdef}
\end{align}
Here,
\[ G(\x,t) = \frac{e^{-\|\x\|^2/4t}}{(4 \pi t)^{d/2}}  \]
is the fundamental solution of the heat equation in $d$ dimensions.
The functions $J[f]$ and $V[F]$ are  referred to as 
{\em initial} (heat) potentials and {\em domain} (heat) potentials,
respectively.  In the remainder of this paper, we assume $d=2$.

For the Neumann problem \eqref{heateqfree}, \eqref{neumannbc}, the classical
representation \cite{guentherlee,pogorzelski} takes the form
\begin{equation}
u(\x,\Delta t) = J[f](\x,\Delta t) + V[F](\x,\Delta t) + S[\sigma](\x,\Delta t)
\label{nsol}
\end{equation}
where
\begin{align}
S[\sigma](\x,\Delta t) &= 
\int_0^{\Delta t} \int_{\Gamma} G(\x-\y,\Delta t-\tau) \sigma(\y,\tau) \,\, ds_{\y} d\tau
\label{spotdef}
\end{align}
is  a {\em single layer} (heat) potential. 
For the Dirichlet problem \eqref{heateqfree}, \eqref{dirichletbc}, the classical
representation takes the form
\begin{equation}
u(\x,\Delta t) = J[f](\x,\Delta t) + V[F](\x,\Delta t) + D[\mu](\x,\Delta t)
\label{dsol}
\end{equation}
where
\begin{align}
D[\mu](\x,\Delta t) &= 
\int_0^{\Delta t} \int_{\Gamma} \frac{\partial G}{\partial n_{\y}} (\x-\y,\Delta t-\tau) \mu(\y,\tau) \,\, ds_{\y} d\tau
\label{dpotdef}
\end{align}
is  a {\em double layer} (heat) potential. Here, 
$\frac{\partial}{\partial n_{\y}}$ denotes the
derivative in the outward normal direction at the boundary point $\y$.
The only unknowns in the representations \eqref{nsol}, \eqref{dsol} are the 
scalar densities $\sigma$ and $\mu$ supported on $\Gamma$. These are obtained 
by solving integral equations to enforce the desired boundary conditions
\cite{guentherlee,pogorzelski}.
Once $\sigma$ or $\mu$ is known, \eqref{nsol}, \eqref{dsol} can be used to 
evaluate the solution at time $t=\Delta t$. This yields a one-step marching 
method for the heat equation that is unconditionally stable 
(see, for example, 
\cite{Arnold89,Brown89,costabel,dargushbanerjee,greengard_lin,powerheat,GLiheatquad,veerapanenibiros1,veerapanenibiros2}).

For our present purposes,
we assume that $\sigma$ and $\mu$ are given.
We assume also that a suitable $M$-stage quadrature has been applied to 
$V[F](\x,\Delta t)$, $S[\sigma](\x,\Delta t)$ and $D[\sigma](\x,\Delta t)$ with 
respect to the time variable, yielding:
\begin{align}
V[F](\x,\Delta t) &\approx \sum_{j=1}^M  w_{V,j} 
\int_{\Omega} G(\x-\y,\Delta t-\tau_j) F(\y,\tau_j) \, d\y \, , \nonumber \\
S[\sigma](\x,\Delta t) &\approx \sum_{j=1}^M  w_{S,j} 
\int_{\Gamma} G(\x-\y,\Delta t-\tau_j) \sigma(\y,\tau_j) \, ds_{\y} \, ,
\label{discretizedpots} \\
D[\mu](\x,\Delta t) &\approx \sum_{j=1}^M  w_{D,j} 
\int_{\Gamma} \frac{\partial G}{\partial n_{\y}}(\x-\y,\Delta t-\tau_j) 
\mu(\y,\tau_j) \, ds_{\y} \, , \nonumber
\end{align}
where $w_{V,j}$, $w_{S,j}$, $w_{D,j}$ are known quadrature weights.

Thus, the computational burden of time-marching 
(that is, evaluating the various heat potentials) is dominated by 
the volume integrals
\begin{equation}
 {\cal V}[f](\x) =
\int_{\Omega} e^{-\frac{|\x-\y|^2}{\delta}} \tilde{f}(\y) \, d\y
     \label{eq:volint}
\end{equation}
and the boundary integrals
\begin{equation}
\begin{split}
 {\cal S}[\sigma](\x) &=
\int_{\Gamma} e^{-\frac{|\x-\y(s)|^2}{\delta}} \tilde{\sigma}(\y(s)) \, ds_{\y} \, , \\
 {\cal D}[\mu](\x) &=
\int_{\Gamma} \frac{\partial}{\partial n_{\y(s)}} 
e^{-\frac{|\x-\y(s)|^2}{\delta}} \tilde{\mu}(\y(s)) \, ds_{\y}  \, ,
\end{split}
\label{eq:bdryint}
\end{equation}
for various values of $\delta$ and 
given functions $\tilde{f}, \tilde{\sigma}, \tilde{\mu}$.
Evaluating these integrals accurately and
efficiently is the focus of the present paper.

\begin{definition}
The integrals \eqref{eq:volint} and 
\eqref{eq:bdryint} will be referred to as 
volume and boundary Gauss transforms, respectively. 
\end{definition}

\begin{definition}
By the {\em discrete Gauss transform} (DGT), we mean
the evaluation of the Gaussian ``potential" at $M$ points
$\{\x_i\}$ due to $N$ sources located at $\{\y_j\}$ of
strength $\{q_j\}$:
\begin{equation}
 F(\x_i) =
\sum_{j=1}^N q_j\cdot e^{-\frac{|\x_i-\y_j|^2}{\delta}}\;\; 
{\rm for}\ i=1,\cdots,M.
     \label{eq:ptfgt}
\end{equation}
\end{definition}

A variety of algorithms have been developed for the rapid evaluation of 
sums of the form \eqref{eq:ptfgt},
such as the fast Gauss transform (FGT) \cite{greengard_strain1}
(see also \cite{greengard98,pfgt,fgt3,tausch-fgt}). While the naive
DGT requires $O(MN)$ work, the FGT
permits the evaluation of the values $\{ F(\x_i) \}$ 
using only $O(M+N)$ work, independent of $\delta$.
High-dimensional versions of the FGT are of 
interest in statistical and machine learning applications
(see, for example, \cite{elgammal03}), but we
are concerned here with physical modeling, where
the ambient dimension is generally less than or equal to three.

Here, we seek to develop a robust version of the FGT that is
fully adaptive, insensitive to $\delta$, and able
to compute transforms with discrete sources,
volume sources and densities supported on boundaries (Fig. \ref{sourcedist_fig}).
Some notable prior work on continuous (volume) 
fast transforms includes \cite{strain_adapheat}, which describes
a triangulation-based adaptive refinement method 
and \cite{veerapanenibiros2}
which makes use of a high-order, 
adaptive quad-tree based discretization. In both cases,
a ``single-level" FGT was superimposed in order to 
achieve linear scaling.
Our approach also relies on an adaptive quad-tree with 
high-order Chebyshev grids on leaf nodes, but we 
carry out a modified version of the FGT on the 
quad-tree itself.
This requires 
a somewhat more 
complicated implementation, following that of the 
hierarchical fast multipole method (FMM) \cite{MLFMM} or, more precisely,
its level-restricted variants developed to compute 
elliptic volume potentials \cite{AskhamCerfon,Ethridge2001,GLee98,Malhotra2016}.
Our hierarchical FGT permits the inclusion of
boundary Gauss transforms (and discrete sources) at the same time.
We should note that
adaptive FGT variants using an FMM data structure have been constructed 
previously, such as in
\cite{Lee2006}, but for the discrete setting only - where small values of
$\delta$ pose no additional quadrature challenges. We also introduce
new error estimates that are relevant for the hierarchical processing.

\begin{figure}[htbp]
\centering
\includegraphics[width=.8\textwidth]{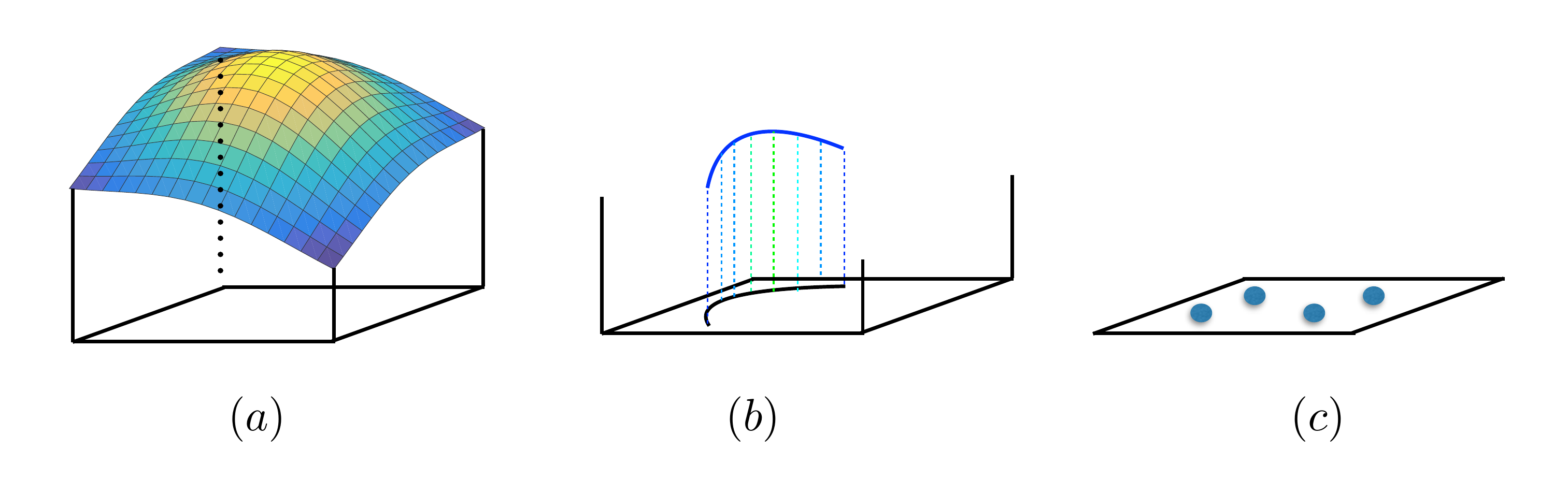}
\caption{We seek a version of the FGT that is able to handle
volume sources (a), densities supported on boundaries (b), and point sources (c).} 
\label{sourcedist_fig}
\end{figure}

The paper is organized as follows. In sections \ref{secdatastruct} and 
\ref{sec_analysis},
we review our adaptive discretization strategy and the analytic machinery on
which the fast Gauss transform is based.
In section \ref{newFGT}, we describe the new 
FGT itself, focusing primarily on the volume integral case
(\ref{eq:volint}), with a brief discussion of the modifications needed for 
(\ref{eq:ptfgt}) and (\ref{eq:bdryint}). We also discuss the incorporation of 
periodic boundary conditions. Section \ref{sec_results} illustrates the 
performance of the algorithm
with several numerical examples and section \ref{sec_conclusions}
contains some concluding remarks.

\section{Data structure} \label{secdatastruct}

\begin{figure}[htbp]
\centering
\includegraphics[width=.8\textwidth]{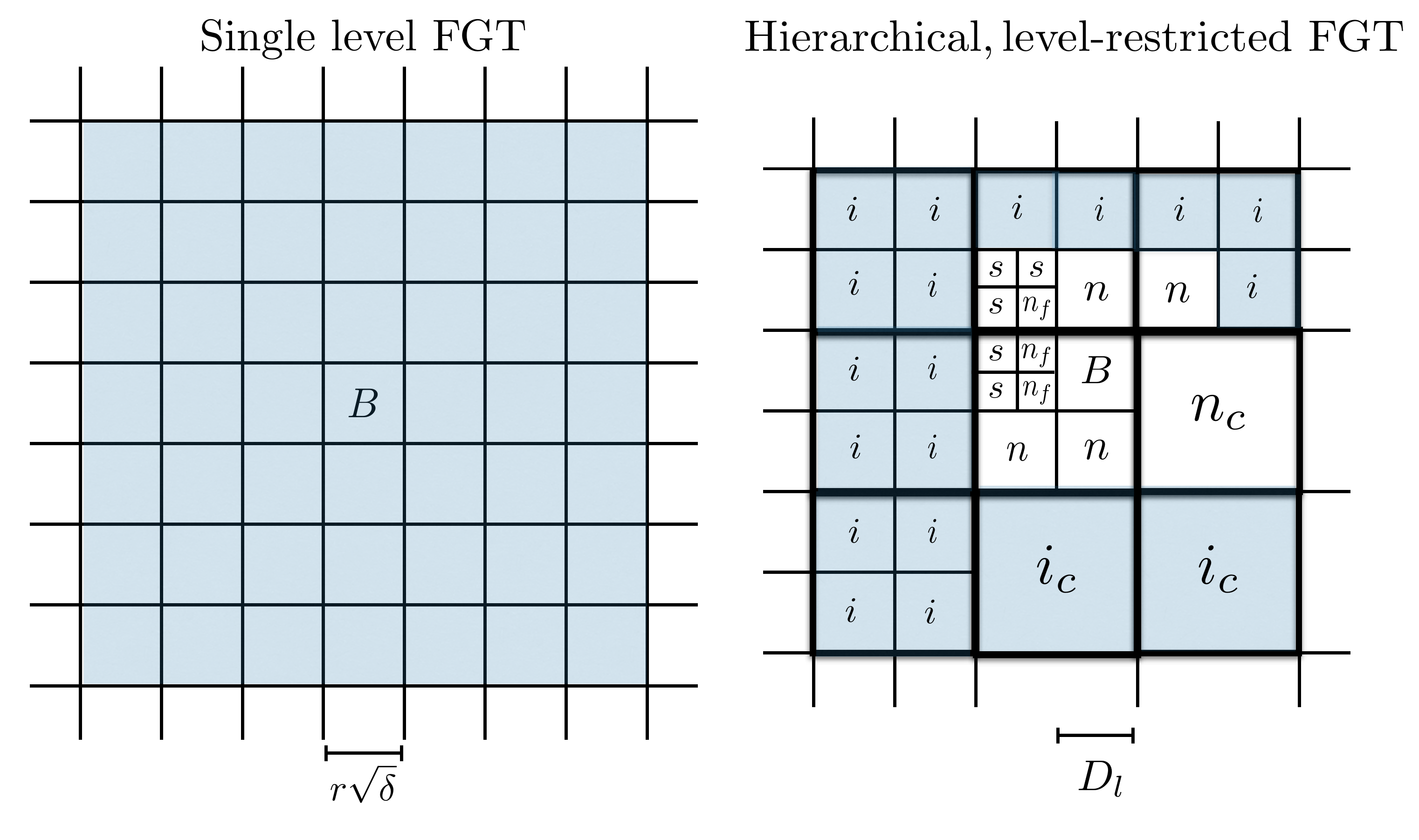}
\caption{The original FGT data structure (left)
and a level-restricted quad-tree data structure (right). 
In the original FGT, the interaction region (shaded blue-gray) consists
of boxes which are close enough to $B$ that the Gaussian field induced by the 
sources in $B$ is significant. (Outside the shaded region, the field is
exponentially small and can be ignored for any fixed precision.)
In the quad-tree, multiple types of interactions must be accounted for,
described in detail in section \ref{newFGT}.
The interaction
list for a typical node $B$ consists of the boxes labeled [$i$], while 
the near neighbors at the same refinement level are labeled [$n$].
For a leaf node $B$,
there can be near neighbors at one coarser
level [$n_c$] or at one finer level [$n_f$] as well. 
The boxes labeled [$s$] are separated 
from $B$ but at a finer level (see Definition \ref{ldef}), while the boxes
labeled [$i_c$] are separated 
from $B$ but at a coarser level 
(see Definitions \ref{ldef} and \ref{ldef2}).} 
\label{fgtfmm}
\end{figure}

In the classical FGT \cite{greengard_strain1}, aimed at the computation of 
(\ref{eq:ptfgt}), where the sources are discrete, 
a uniform grid is superimposed on the computational domain, with a 
box size of dimension $(r \sqrt{\delta})^d$, where $r \approx 1$
(Fig. \ref{fgtfmm}, left).
Because of the exponential decay of the Gaussian, it is easy
to see that only a finite range of nearby boxes need to be considered 
to achieve any desired precision. That is, the effect due to sources in 
$B$ at targets that are at least $m$ boxes away is of the order 
$O(e^{-m^2r^2})$. Since the field due to sources in any box $B$ is efficiently
represented by a suitable Hermite expansion (see section \ref{sec_analysis}),
it is straightforward to develop an algorithm of complexity
$O(N+M)$, where  $N$ is the number of 
discrete sources and $M$ is the number of targets. 
The FGT is easily modified to allow for adaptivity.
One simply needs to sort 
the source and target points on the uniform grid while
ignoring empty boxes and keeping track of the relevant neighbors for each box.
The total storage is then of the order $O(N+M)$ as well.
This can be accomplished, for example, with an adaptive quad-tree that is refined
uniformly to a level where the box size is approximately $(r \sqrt{\delta})^d$, 
pruning empty boxes on the way.

Such a strategy fails for volume integrals of the form
(\ref{eq:volint}), since there are no empty boxes.
Instead, we will assume that the right-hand side (the function $f$ in 
\eqref{eq:volint}) is specified on a level-restricted quad-tree.  These data
structures have 
been shown to be extremely effective for elliptic volume integrals
\cite{AskhamCerfon,Huang2006,Ethridge2001,Langston2011,GLee98,Malhotra2016}.
For the sake of simplicity, we assume that the source distribution 
$f$ in (\ref{eq:volint}) is supported in 
the unit box $D$, centered at the 
origin. Following the discussion of \cite{Ethridge2001}, we assume that
superimposed on $D$ is a hierarchy of refinements (a quad-tree). 
Grid level 0 is defined to be $D$ itself, 
with grid level $l+1$ obtained recursively by subdividing each box at level $l$ 
into four equal parts. 
If $B$ is a fixed box at level $l$, the four boxes at level $l+1$ obtained 
by its subdivision will be referred to as its children. In  a level-restricted,
adaptive tree, we do not assume the same number 
of levels is used in all subdomains of $D$.  We do, however, require that
two leaf nodes which share a boundary point must be no more than one refinement level 
apart (Fig. \ref{fgtfmm}, right).

On each leaf node $B$, we assume that we are given $f$ on a $k\times k$ tensor product 
grid. We may then construct a $k$th-order polynomial 
approximation to $f$ on $B$ of the form
\begin{equation}
 f_B(y_1,y_2) \approx \sum_{j=1}^{N_k}  c_B(j) \, b_j(y_1,y_2),
  \label{eq:leafnode}
\end{equation}
\noindent
where $N_k = \frac{k(k+1)}{2}$ is the number of basis functions needed for $k$th 
order accuracy, and the basis functions $b_j$ are assumed to be scaled to the 
relevant box size and centered on the box center. The coefficient vector
is defined to be $\vec c_B = (c_B(1),\dots,c_B(N_k))$.

For fourth or sixth order accuracy, one can use as
basis functions 
\[ \{ y_1^l y_2^m |\ l,m \geq 0, l+m \leq k-1 \}.  \]
If we let $\vec{f_B} \in {\bf R}^{k^2}$ denote the given function
values (in standard ordering),
then the coefficient vector $\vec{c_B}$ can be computed as the solution of a 
least squares problem (interpolating the desired data $\vec{f_B}$ at the 
corresponding points). The solution operator for this least squares
task is denoted by 
${\cal P} \in {\bf R}^{N_k \times k^2}$, so that
\[ \vec c_B = {\cal P} \, \vec{f_B}. \]
${\cal P}$ can be precomputed and stored (say, using QR factorization).
For eighth (or higher) order accuracy, 
polynomial interpolation is stabilized by assuming that $f$ is given on 
a $k \times k$ tensor product Chebyshev grid and using as basis functions 
\[ \{ T_l(y_1) T_m(y_2) |\ l,m \geq 0, l+m \leq k-1 \},  \]
where $T_l(x)$ denotes the (suitably scaled) Chebyshev polynomial of degree $l$.
The coefficients of the tensor product Chebyshev expansion can be computed 
efficiently using the fast cosine transform \cite{boyd}.
 
In order to develop a fast algorithm for the various 
kinds of source distributions shown in Fig. \ref{sourcedist_fig}, we 
will make use of efficient far field and local representations of the induced field.

\section{Analytical apparatus} \label{sec_analysis}


Following the discussion in \cite{greengard_strain1},
we define the Hermite functions $h_n(x)$ by 
\[ 
h_n(x)=(-1)^n D^n e^{-x^2},\;\;\;x\in \bbR,
\]
where $D=d/dx$. They satisfy the relation
\begin{equation}
e^{-(x-y)^2/\delta}=\sum_{n=0}^{\infty} \frac{1}{n!}\left(\frac{y-y_0}{\sqrt{\delta}}\right)^n h_n\left(\frac{x-y_0}{\sqrt{\delta}}\right),
\label{mpole1d}
\end{equation}
where $y_0\in \bbR$ and $\delta>0$.
This formula can be interpreted as an Hermite expansion centered at $y_0$ for the 
Gaussian field $e^{-(x-y)^2/\delta}$ at the target $x$
due to the source at $y$.
Interchanging $x$ and $y$ one can also write:
\begin{equation}
e^{-(x-y)^2/\delta}=\sum_{n=0}^{\infty} \frac{1}{n!}
h_n\left(\frac{y-x_0}{\sqrt{\delta}}\right) 
\left(\frac{x-x_0}{\sqrt{\delta}}\right)^n.
\label{locexp1d}
\end{equation}
This expresses the Gaussian as a Taylor series in the target location $x$ about 
a center $x_0$. 

It will be convenient to use multi-index notation. In two dimensions, a
multi-index is a pair of non-negative integers
$\alpha=(\alpha_1, \alpha_2)$ with which, for any
$\x = (x_1,x_2) \in{\bf R}^2$, we define:
\[
|\alpha| =\alpha_1+\alpha_2, \quad \alpha! =\alpha_1!\alpha_2!,  \quad
\x^{\alpha} =x_1^{\alpha_1} x_2^{\alpha_2}, \quad
D^{\alpha} =\partial_{x_1}^{\alpha_1}\partial_{x_2}^{\alpha_2},
\]
If $p$ is an integer,
we say $\alpha \geq p$ if $\alpha_1,\alpha_2 \geq p$.
Multi-dimensional Hermite functions are defined by
\[
h_{\alpha}(\x)=h_{\alpha_1}(x_1)h_{\alpha_2}(x_2),
\]
and the analogs of \eqref{mpole1d} and \eqref{locexp1d} are
\begin{equation}
e^{-|\x-\y|^2/\delta}= \frac{1}{\alpha!}
\sum_{\alpha\geq 0} \left( \frac{\y-\y_0}{\sqrt{\delta}} \right)^{\alpha} 
h_{\alpha} \left( \frac{\x-\y_0}{\sqrt{\delta}} \right),
\label{mpole2d}
\end{equation}
\begin{equation}
e^{-|\x-\y|^2/\delta}= \frac{1}{\alpha!}
\sum_{\alpha\geq 0} 
h_{\alpha} \left( \frac{\y-\x_0}{\sqrt{\delta}} \right)
\left( \frac{\x-\x_0}{\sqrt{\delta}} \right)^{\alpha}. 
\label{locexp2d}
\end{equation}

\subsection{Hermite expansions and translation operators}

We turn now to the analytical apparatus needed in the FGT algorithm. 
The first lemma describes how to transform the field due to a volume 
source distribution
and a collection of discrete Gaussians
into an Hermite expansion about the center of box $B$ in which they are supported.

\begin{lemma}\label{lemma:hexp}
Let B be a box with center $\s_B$ and side length $r\sqrt{\delta}$ and let
the Gaussian field $\phi(\x)$ be defined by 
\begin{equation}  
\phi(\x)=\int_{B}e^{-\frac{|\x-\y|^2}{\delta}} f(\y)\, d\y \ +
\sum_{j=1}^{N_s} 
q_j e^{-\frac{|\x-\y_j|^2}{\delta}} \, 
  \label{eq:volintb}
\end{equation}
where the $\y_j$ lie in B.
Then,
\begin{equation}
\phi(\x)=
\sum_{\alpha\geq 0}A_{\alpha} h_{\alpha}\left(\frac{\x-\s_B}{\sqrt{\delta}}\right),
  \label{eq:hexp}
\end{equation}
where 
\begin{equation}
A_{\alpha}=
\frac{1}{\alpha!} \left( 
\int_{B} \left(\frac{\y-\s_B}{\sqrt{\delta}}\right)^{\alpha}f(\y)\, d\y + 
 \sum_{j=1}^{N_s} \left(\frac{\y_j-\s_B}{\sqrt{\delta}}\right)^{\alpha} q_j 
\right) \, .
  \label{eq:hcoeff}
\end{equation}
The error in truncating the Hermite expansion with $p^2$ terms is given by
\begin{equation}
 |E_H(p)|=\left|\sum_{\alpha\geq p}A_{\alpha} 
h_{\alpha}\left(\frac{\x-\s_B}{\sqrt{\delta}}\right)\right| \leq K^2Q_B(2S_r(p)+T_r(p))T_r(p),
  \label{eq:E_H(p)}
\end{equation}
where
\begin{equation}
Q_B=\int_B |f(\y)| d\y + 
\sum_{j=1}^{N_s} |q_j|,
  \label{eq:Q_B}
\end{equation}
\begin{equation}
S_r(p)=\sum_{n=0}^p \frac{r^n}{\sqrt{n!}},\;\;\;T_r(p)=\sum_{n=p}^{\infty} \frac{r^n}{\sqrt{n!}}.
  \label{eq:srtr}
\end{equation}
and $K < 1.09$.
\end{lemma}

\begin{proof}
The error estimate relies on Cramer's inequality, which takes the form
\begin{equation} \label{eq:cramer}
\frac{1}{\alpha!}|h_{\alpha}(\x)|\leq K^2 2^{|\alpha|/2}
\frac{1}{\sqrt{\alpha!}}e^{-|\x|^2/2}
\end{equation}
in two dimensions, where $K < 1.09$, and the fact that
\begin{equation}
\left|\frac{\y-\s_B}{\sqrt{\delta}}\right|\leq \frac{\sqrt{2}}{2}r,
\end{equation}
for $\y$ in $B$.
The desired result follows from integration over the domain $B$ and summation
over the discrete sources.
\end{proof}

Note that the Hermite expansion converges
extremely rapidly for $r < 1$.  
For larger $r$, they still converge but require larger values of $p$.
(See \cite{baxter,Lee2006,Strain1991,wankarniadakis} for further discussion
of error estimates.)
In the original FGT,
setting $r \approx 1$ is a sensible choice, since a modest value of $p$ 
is sufficient and the number of boxes within the interaction region 
(where the Gaussian field is not vanishingly small) is modest as well.
The interaction region for a box $B$ is the shaded area on the left 
in Fig. \ref{fgtfmm}. 
A second thing to note is that the estimate is {\em uniform}
with respect to the target. In the original FGT, this is necessary since the Hermite 
expansion is evaluated at all relevant locations.
In the hierarchical FGT, however, the Gaussian field due to an Hermite expansion
is evaluated only for boxes that are ``well-separated"
(the boxes labeled by $i$ on the right-hand side of Fig. \ref{fgtfmm}). 
Moreover, we will compute such interactions at every level of the quad-tree, 
so that the boxes could be of arbitrary size. Fortunately, once boxes are separated
by a distance $R \sqrt{\delta}$, their interactions can be ignored with an error
of the order $O(e^{-R^2})$, limiting the size of the expansions
(see Fig. \ref{ntermsfig}).

\begin{figure}[htbp]
\centering
\includegraphics[width=.95\textwidth]{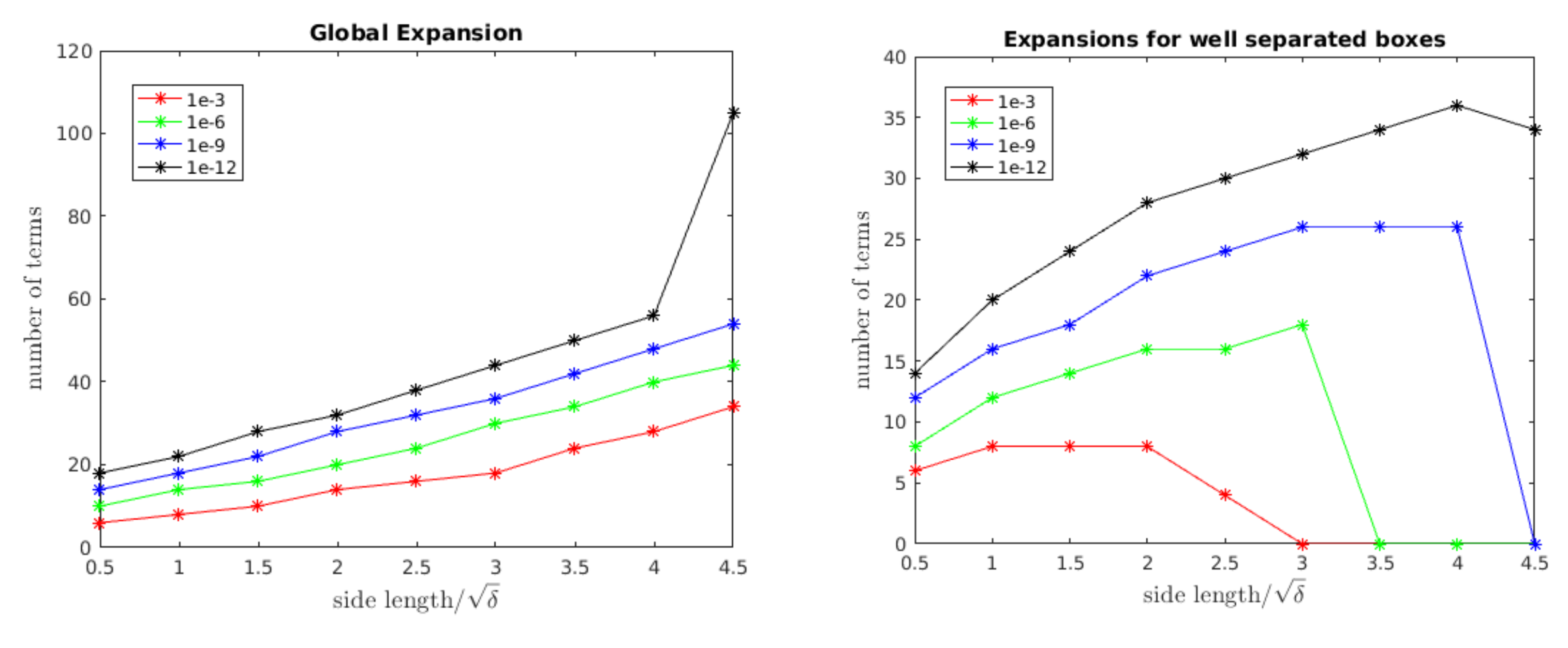}
\caption{(Left) plot of the number of terms $p$ needed in an Hermite
expansion as a function of the box size, when used uniformly in the 
plane, including the near field. (Right)
plot of the number of terms needed when used only in the far field.  } 
\label{ntermsfig}
\end{figure}

In the full algorithm, a more refined estimate that
makes use of the separation criterion will be useful.
We have the following lemma.

\begin{lemma}\label{lemma:hexpp}
Let $B$ be a box with center $\s_B$ and side length $r\sqrt{\delta}$, and let
$C$ be a box with center $\t_C$ and side length $r\sqrt{\delta}$ with 
$x \in C$.
Assuming the distance between B and C is at least $r\sqrt{\delta}$, 
the Gaussian field defined by \eqref{eq:volintb} and its Hermite expansion
\eqref{eq:hexp} satisfy the error bound:
\begin{equation}
|E_H(p)|=\left|\sum_{\alpha\geq p}A_{\alpha} 
h_{\alpha}\left(\frac{\x-\s_B}{\sqrt{\delta}}\right)\right|=
K^2Q_B e^{-\frac{9}{8}r^2}(2S_r(p)+T_r(p))T_r(p),
  \label{eq:E_H(p)_sep}
\end{equation}
where $Q_B$ is given by (\ref{eq:Q_B}), and $S_r(p)$, $T_r(p)$ are given by 
(\ref{eq:srtr}).
\end{lemma}

\begin{proof}
The proof follows the same outline as that of Lemma \ref{lemma:hexp}. 
In this case, after applying Cramer's inequality (\ref{eq:cramer}), 
we make the additional observation that $|\frac{\x-\s_B}{\sqrt{\delta}}|\geq \frac{3}{2}r$, which contributes to the exponential decay in $r$.
\end{proof}

The next lemma describes the conversion of an Hermite expansion about $\s_B$ 
into a Taylor expansion about $\t_C$.

\begin{lemma}\label{lemma:texp} 
Let
\begin{equation}
\phi(\x)=\sum_{\alpha\geq 0}A_{\alpha} h_{\alpha}\left(\frac{\x-\s_B}{\sqrt{\delta}}\right),
  \label{eq:hexp2}
\end{equation}
denote the Hermite expansion of a 
Gaussian field induced by a source distribution in a box $B$ with center $\s_B$ 
and side length $r\sqrt{\delta}$. Then
$\phi(\x)$
has the following Taylor expansion about the center $\t_C$ of box $C$ 
with side length $r\sqrt{\delta}$:
\begin{equation}
\phi(\x)=
\sum_{\beta\geq 0} B_{\beta}\left(\frac{\x-\t_C}{\sqrt{\delta}}\right)^{\beta}.
\label{eq:texp}
\end{equation}
The coefficients are given by
\begin{equation}
B_{\beta}=\frac{(-1)^{|\beta|}}{\beta!}\sum_{\alpha\geq 0} A_{\alpha}h_{\alpha+\beta}\left(\frac{\s_B-\t_C}{\sqrt{\delta}}\right).
\end{equation}
Assuming that the distance between boxes $B$ and $C$ is at least 
$r\sqrt{\delta}$, the error $E_T(p)$ in truncating the Taylor series after
$p^2$ terms satisfies
\begin{equation}
|E_T(p)|=\left|\sum_{\alpha\geq p}A_{\alpha} 
h_{\alpha}\left(\frac{\x-\s_B}{\sqrt{\delta}}\right)\right| \leq 
K^2Q_B e^{-\frac{9}{8}r^2}(2S_r(p)+T_r(p))T_r(p),
  \label{eq:E_T(p)}
\end{equation}
where $Q_B$ is given by (\ref{eq:Q_B}), and $S_r(p)$, $T_r(p)$ are given by (\ref{eq:srtr}).
\end{lemma}

\begin{proof}
This result follows, again, from 
the standard error estimate in \cite{baxter,Strain1991,wankarniadakis}, 
with one modification; the exponential
term in Cramer's inequality can be bounded by 
$e^{-\frac{9}{8}r^2}$, instead of 1. 
\end{proof}

In practice, we need a variant of Lemma \ref{lemma:texp},
in which the Hermite expansion is truncated before being converted to a Taylor 
expansion. 
\begin{lemma}\label{lemma:H2L}
Let 
\begin{equation}
\phi(\x)=\sum_{\alpha\leq p}A_{\alpha} h_{\alpha}\left(\frac{\x-\s_B}{\sqrt{\delta}}\right)
\end{equation}
denote a truncated Hermite expansion
corresponding to the Gaussian field induced by a source distribution in a box $B$ 
with center $\s_B$ and side length $r\sqrt{\delta}$.
The induced Taylor series in 
a box $C$ with center $\t_C$ and side length $r\sqrt{\delta}$ is given by
\begin{equation}
\phi(\x)=\sum_{\beta\geq 0} 
C_{\beta}\left(\frac{\x-\t_C}{\sqrt{\delta}}\right)^{\beta},
  \label{eq:texptrunc}
\end{equation}
with coefficients
\begin{equation}
C_{\beta}=\frac{(-1)^{|\beta|}}{\beta!}\sum_{\alpha\leq p} A_{\alpha}h_{\alpha+\beta}\left(\frac{\s_B-\t_C}{\sqrt{\delta}}\right).
  \label{eq:tcoeff}
\end{equation}

Assuming that the distance between boxes $B$ and $C$ is at least $r\sqrt{\delta}$, 
the error $E_{HT}(p)$ in truncating the Taylor series after
$p^2$ terms satisfies the bound
\begin{equation}
|E_{HT}(p)|=\left|\sum_{\alpha\geq p}A_{\alpha} 
h_{\alpha}\left(\frac{\x-\s_B}{\sqrt{\delta}}\right)\right| \leq
K^2Q_B e^{-2r^2}(2S_r(p)+T_r(p))T_r(p)S_r^2(p),
  \label{eq:E_{HT}(p)}
\end{equation}
where $Q_B$ is given by (\ref{eq:Q_B}), and $S_r(p)$, $T_r(p)$ are given by (\ref{eq:srtr}).
\end{lemma}

\begin{proof}
The result is a straightforward application of the triangle inequality and Lemma
\ref{lemma:texp}. 
\end{proof}

Note that the total error in using both an Hermite and a local
expansion consists of two contributions: the first 
comes from truncating the Hermite expansion,
given by (\ref{eq:E_H(p)}), while the second comes from truncating 
the local expansion, according to \eqref{eq:E_{HT}(p)}.

For the 
hierarchical FGT, we will also need to propagate Hermite and Taylor expansions 
between levels of the quad-tree. 
The following two lemmas provide the needed analytical tools. 
Lemma \ref{lemma:H2H} describes a formula for shifting the center of
an Hermite expansion, and Lemma \ref{lemma:L2L} describes one for shifting the 
center of a Taylor expansion. The derivation is straightforward \cite{Lee2006}.

\begin{lemma}\label{lemma:H2H}
Let a Gaussian field be given by 
the Hermite expansion
\begin{equation}
\phi(\x)=\sum_{\alpha\geq 0} 
A_{\alpha}h_{\alpha}\left(\frac{\x-\s_B}{\sqrt{\delta}}\right),
\end{equation}
about a center $\s_B$ and let $\s_C$ denoted a shifted expansion center.
Then,
\begin{equation}
\phi(\x)=\sum_{\beta\geq 0} 
B_{\beta}h_{\beta}\left(\frac{\x-\s_C}{\sqrt{\delta}}\right),
\end{equation}
where the coefficients are given by:
\begin{equation}
B_{\beta}=\sum_{\alpha\leq\beta} \frac{\alpha!}{\beta!}
\binom{\beta}{\alpha}
\left(\frac{\s_B-\s_C}{\sqrt{\delta}}\right)^{\beta-\alpha}A_{\alpha}.
  \label{eq:H2H}
\end{equation}
\end{lemma}

\begin{lemma}\label{lemma:L2L}
Let $\t_B \in \bbR^2$ and let 
$\{C_{\alpha}\}$ denote the expansion coefficients for a truncated Taylor
series with $p^2$ terms. Letting 
$\t_C \in  \bbR^2$ be a shifted expansion center, we have 
\begin{equation}
\sum_{\alpha\leq p} C_{\alpha} \left(\frac{\x-\t_B}{\sqrt{\delta}}\right)^{\alpha}=\sum_{\beta\leq p} C'_{\beta} \left(\frac{\x-\t_C}{\sqrt{\delta}}\right)^{\beta},
\end{equation}
where 
\begin{equation}
C'_{\beta}=\sum_{\beta\leq\alpha\leq p} C_{\alpha} 
\binom{\alpha}{\beta} \left(\frac{\t_C-\t_B}{\sqrt{\delta}}\right)^{\alpha-\beta} \, .
  \label{eq:L2L}
\end{equation}
\end{lemma}

\subsection{Local interactions}

In the previous section, we summarized the analytical machinery 
needed for the fast evaluation of Gaussian fields with well separated sources 
and targets. Before providing a formal description of the full algorithm
in the next section, it remains 
to consider the computation of local interactions between neighboring boxes
at the level of leaf nodes. For point sources, this is done by direct
evaluation. We concentrate in this section on domain integrals, and defer 
a discussion of densities supported on boundaries to section \ref{bdryfgt}.

\begin{figure}[htbp]
\centering
\includegraphics[width=.5\textwidth]{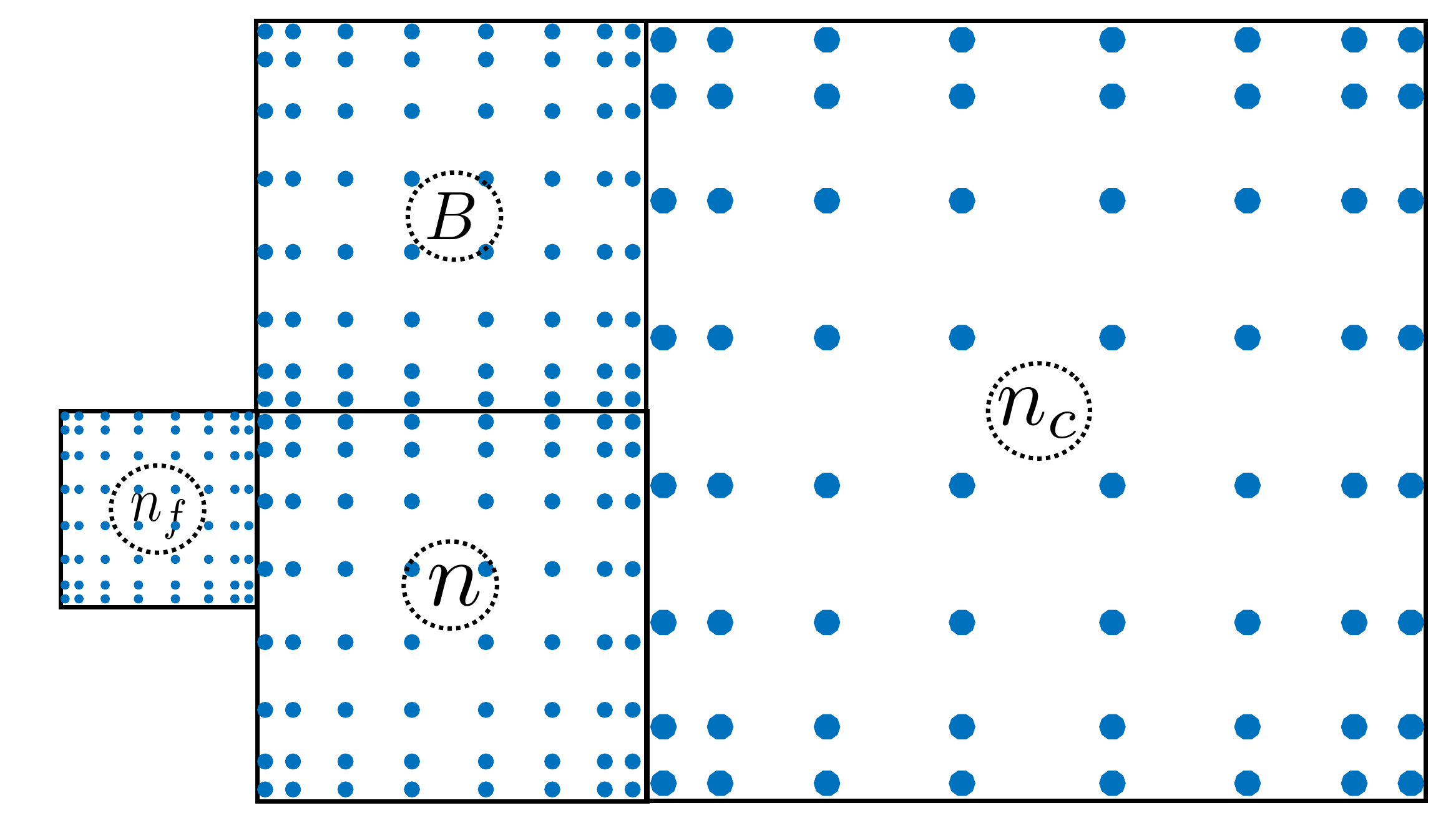}\label{nborgrids}
\caption{For a leaf node $B$, there are three types of possible local 
interactions: the interaction with a colleague 
(a neighbor at the same refinement level, including the self-interaction),
and the interaction with fine and coarse neighbors - one level finer or one 
level coarser, respectively. The grid shown corresponds to an 
8$th$ order accurate tensor product Chebyshev discretization.} 
\label{nborsfig}
\end{figure}

Thus, suppose $B$ is a leaf node - that is a box at level $l$ of the tree 
hierarchy on which a $k \times k$ tensor product grid of function values 
has been specified. Let $r_l$ denote the side length of $B$, so that its
area is $r_l \times r_l$.
Consider now a target point $\t$, which lies in either $B$, a neighboring box of $B$
at the same refinement level, or a coarse or fine neighbor for $B$, 
which can be at most one refinement level apart.
Because of the translation invariance of the kernel,
a simple counting argument shows that there are at most $k \times k \times 9$ 
possible targets at the same level and at most 
$k \times k \times 12$  possible targets in neighbors at either a coarser or 
finer level.
Recalling that the source distribution $f$ on $B$ is given by
(\ref{eq:leafnode}), the Gaussian field induced at $\t$ by
$f_B$ can be approximated by
\begin{equation}
\psi_B(\t)=\sum_{n=1}^{N_k} c_B(n)G(\t,n),
  \label{eq:local}
\end{equation}
with
\begin{equation}
G(\t,n) = \int_B e^{-\frac{(t_1-y_1)^2}{\delta}} 
e^{-\frac{(t_2-y_2)^2}{\delta}} 
b_n(y_1,y_2) \, dy_1 dy_2,
  \label{eq:localweights}
\end{equation}
where $(t_1,t_2)$, $(y_1,y_2)$ denote the coordinates of the target $\t$ 
and source with respect to the center of box $B$. 
Once the values $\{ G(\t,n) \}$ have been tabulated for all possible 
target locations and all basis functions, {\em all} local interactions 
can be computed directly from \eqref{eq:local}, with some care in bookkeeping.

Assuming $b_n(y_1,y_2) = p_{n_1}(y_1) p_{n_2}(y_2)$,
the formula for $G(\t,n)$ can be written in the form
\begin{equation}
G(\t,n) = \int_{-r_l/2}^{r_l/2} e^{-\frac{(t_1-y_1)^2}{\delta}} 
p_{n_1}(y_1) dy_1 \cdot 
\int_{-r_l/2}^{r_l/2} e^{-\frac{(t_2-y_2)^2}{\delta}} 
p_{n_2}(y_2) dy_2 .
\label{eq:localweights_sep}
\end{equation}
Thus, for colleagues (boxes at the same refinement level), there are at most 
$3k$ possible relative target locations $(t_i-y_i)$ and at most $k$ basis 
functions $p_{n_i}(y_i)$, which are either monomials or scaled 
Chebyshev polynomials.
These $3k^2$ numbers can be computed in milliseconds
on a single core.
For coarse or fine colleagues, it is straightfoward to check that 
there are at most 
$4k$ possible relative target locations $(t_i-y_i)$, so that these tables 
involving $4k^2$ numbers can be generated in milliseconds as well.
Finally, we note that such tables must be generated for each refinement 
level that contains a leaf node.

\section{FGT algorithm} \label{newFGT}

We now describe an adaptive version of the Fast Gauss Transform, closely
following the discussion in
\cite{Ethridge2001}. 
Since the Gaussian kernel $e^{-\|\x-\y\|^2/\delta}$
is rapidly decaying, we will ignore interactions beyond a distance
where they can be considered negligible, according to a user-defined
precision $\epsilon$. That is, we define a cut-off parameter $r_c$
so that 
$e^{-|\x-\y|^2/\delta}\leq \epsilon$, when $\| \x-\y\| \geq r_c\sqrt{\delta}$. 
Clearly, if a source box has side length
greater than or equal to $r_c\sqrt{\delta}$, its contribution to 
well-separated boxes is negligible. 
We will also make use of the following definitions:

\begin{definition} \label{def:lcut}
{\rm (Cutoff level):} Given a quad-tree with levels $l=0,1,\cdots,L$, the cutoff level is defined to be the coarsest level of the tree at which the box size
is smaller than or equal to $r_c\sqrt{\delta}$. We denote this by
$l_{cut}$. If the box size is greater 
than $r_c\sqrt{\delta}$ even at the finest level (level $L$), we 
let $l_{cut}=L+1$. 
\end{definition}

\begin{definition} \label{ldef}
{\rm (Neighbors):} 
Leaf nodes at the same level as $B$ 
which share a boundary point, including $B$ itself, are referred to as
{\rm colleagues}. 
Leaf nodes at the level of 
$B$'s parent which share a boundary point with $B$ are referred to as
the {\rm coarse neighbors} of $B$.
Leaf nodes one level finer than $B$
which share a boundary point with $B$ are referred to as
{\rm fine neighbors}.
Together, the union of the 
colleagues, coarse neighbors and fine neighbors of $B$
are referred to as $B$'s {\rm neighbors}.
The {\em s-list} of a box $B$ consists of those children
of $B$'s colleagues which are not fine neighbors of $B$
(Fig. \ref{fgtfmm}).
\end{definition}

\begin{definition} \label{ldef2}
{\rm (Interaction lists):} 
The {\rm interaction region} for $B$ consists of the area covered by
the neighbors of $B$'s parent, excluding the area covered by
$B$'s colleagues and coarse neighbors. The {\rm interaction list}
for $B$ consists of those boxes in the interaction region which
are at the same refinement level (marked $i$ in Fig.\ref{fgtfmm}),
and is denoted by ${\cal I}(B)$.
Boxes at coarser levels will be referred to
as the {\rm coarse interaction list}, denoted
by ${\cal I}_c(B)$
(marked $i_c$ in Fig.\ref{fgtfmm}).
\end{definition}

\begin{definition}  \label{def:exps}
{\rm (Expansions):} We denote by $B_{l,k}$ the {\rm k}th box at refinement 
level $l$ and by $\Phi_{l,k}$ the Hermite
expansion describing the far field due to the source distribution supported
inside $B_{l,k}$.
We denote by $\Psi_{l,k}$ the local
expansion describing the field due to the source distribution outside
the neighbors of $B_{l,k}$ and by
$\tilde \Psi_{l,k}$ the local
expansion describing the field due to the source distribution outside
the neighbors of the {\rm parent} of $B_{l,k}$.
When the context is clear, we will sometimes use the notation 
$\Phi(B)$, $\Psi(B)$, $\tilde \Psi(B)$ to describe the expansions associated
with a box $B$.
\end{definition}

\begin{remark}
Let $B = B_{l,k}$ be a box in the quad-tree hierarchy 
with children $C_1,C_2,C_3,C_4$.
Then, according to Lemma \ref{lemma:H2H}, 
there is a linear operator ${\cal T}_{HH}$ for which
\begin{equation}
 \Phi_{l,k} =  \Phi(B) = {\cal T}_{HH}[ \Phi(C_1),
\Phi(C_2), \Phi(C_3), \Phi(C_4)]. 
\label{lolo}
\end{equation}
The operator
${\cal T}_{HH}$ is responsible for merging the expansions of four children 
into a single expansion for the parent. 
Likewise, according to Lemma \ref{lemma:L2L} there is a linear operator ${\cal T}_{LL}$ for which
\begin{equation}
 [\tilde \Psi(C_1), \tilde \Psi(C_2), \tilde \Psi(C_3), \tilde \Psi(C_4)]
 = {\cal T}_{LL} \, \Psi_{l,k} = 
 {\cal T}_{LL} \, \Psi(B).
\label{tata}
\end{equation}
${\cal T}_{LL}$ is responsible for shifting the incoming data 
(the local expansion) from a parent box to its children.
Finally, according to Lemma \ref{lemma:H2L}, 
for any source box  $B_{l',k'}$ in the interaction list ${\cal I}(B)$
of box $B_{l,k}$, there
is a linear operator ${\cal T}_{HL}$ for which 
the induced field in $B_{l,k}$ is given by 
$\Psi = T_{HL} \Phi_{l',k'}$. Clearly,
\begin{equation}
 \Psi_{l,k} = \tilde \Psi_{l,k} + \sum_{i \in {\cal I}(B)} 
T_{HL} \Phi_{i}. 
\label{lota}
\end{equation}
\end{remark}

Since our adaptive algorithm is operating on a level-restricted adaptive tree,
the leaf nodes need to 
handle far field interactions between boxes at different levels.
More precisely, viewing each such leaf node $B$ as a ``target box",
we need to incorporate the influence of the 
{\em s-list}  and 
the coarse interaction list
(see Fig. \ref{fgtfmm}).
For every box in the {\em s-list}, its Hermite expansion is rapidly 
convergent in $B$ and its 
influence can be computed by direct evaluation of the series. 
We also need to compute the dual interaction - namely the influence of
a leaf node $B$ on a box $B'$ in the {\em s-list}.
Rather than evaluate the Hermite expansion of $B$ at all targets in $B'$,
or shifting to a local expansion in $B'$, we can
directly expand the influence of the polynomial source distribution in $B$,
given by the
coefficients $\vec c_{B}$, as a local expansion in $B'$.
Thus, 
incorporating all far field interactions
into (\ref{lota}), we have: 
\begin{equation}
 \Psi_{l,k} = \tilde \Psi_{l,k} + \sum_{i \in {\cal I}} 
T_{HL} \Phi_{i}, 
+ \sum_{i \in {\cal I}_c} T_{direct} \, \vec c_i.
\label{lotanew}
\end{equation}
The operator $T_{direct}$, which maps the coefficients of a polynomial 
approximation of the density in $B'$ (a coarse interaction list box)
onto the $p^2$ coefficients of the local expansion in $B$ 
can be precomputed and stored for each level in the quad-tree hierarchy.
Inspection of Fig. \ref{fgtfmm}, the translation invariance of the kernel,
and a simple counting argument show that
this requires $O(k p L)$ work and storage, where 
$k$ is the order of polynomial approximation,
$p$ is the order of the local expansion, and 
$L$ is the number of levels.
More precisely, let $b_n(y_1,y_2) = p_{n_1}(y_1) \, p_{n_2}(y_2)$
be a basis function for the polynomial approximation in box $B'$
and let $\alpha = (\alpha_1,\alpha_2)$ denote the multi-index of a term
in the induced local expansion in $B$. Then
\begin{equation}
T_{direct}(\alpha,n) = T_1(\alpha_1,n_1) \, T_2(\alpha_2,n_2),
\label{tdirect}
\end{equation}
where 
\[
T_{i}(\alpha_i,n_i) = \frac{1}{\alpha_i!} \int_{D_{l-1}/2}^{D_{l-1}/2}
h_{\alpha_i} \left( \frac{y_i-s_i}{\sqrt{\delta}} \right)
p_{n_i}(y_i)\, dy_i \, ,
\]
$(s_1,s_2)$ denotes the center of $B$,
and $D_{l-1}$ denotes the side length of box $B'$ at level $l-1$.

\subsection{Pseudocode for the FGT} \label{algorithm}

We assume we are given a square domain $B_{0,0}$, 
on which is superimposed an adaptive hierarchical quad-tree with
$l_{max}$  refinement levels. We let $l_{cut}$ denote the cutoff level.
If $l_{cut}\leq l_{max}$, for each level l that satisfies
the condition $l_{cut}\leq l \leq l_{max}$, 
determine the number of terms
needed in the Hermite expansions $N_h(l)$ and the number of terms needed 
in the local expansion $N_t(l)$ according
to the box size, the parameter $\delta$ that defines the variance of the 
Gaussian, the user-defined precision $\epsilon$, and 
the estimates
\eqref{eq:E_H(p)_sep}, \eqref{eq:E_T(p)}.

We denote the leaf nodes by $B_i, \ i = 1,\dots,M$, where
$M$ is the total number of leaf nodes across all levels.
We assume that the source distribution on each $B_i$ is given by a 
collection of point sources, as well as a smooth function $f$, sampled on 
a $k\times k$ grid. The number of grid points is denoted by
$N = Mk^2$ and the number of discrete sources is denoted by $N_s$.
We assume the output is desired at the $N$ grid points as well as the
$N_s$ source locations. 

\newpage

\def\step#1{\par \vspace{.05in} \begin{center} {\bf #1} \end{center} }

              \step{Step I: Upward pass}
\begin{tabbing}
.......\=.......\=.......\=.......\=.......\=      \kill

\> {\bf for $l = l_{max},\dots,l_{cut}$ } \\
\> \>  {\bf for} every box $j$ on level $l$ \\
\> \>  \> {\bf if} $j$ is childless {\bf then} \\
\> \> \> \>  $\bullet$ Form Hermite expansion 
             $\Phi_{l,j}$ using (\ref{eq:hcoeff}) \\
\> \> \> {\bf else} \\
\> \> \> \> $\bullet$ Form Hermite expansion 
             $\Phi_{l,j}$ by merging the expansions of \\
\> \> \> \> $\ \ $ its children using ${\cal T}_{HH}$ 
            (see Lemma \ref{lemma:H2H}) \\
\> \> \> {\bf endif} \\
\> \> {\bf end} \\ 
\> {\bf end} \\ 
\end{tabbing}
              \step{Step II: Downward pass}
\begin{tabbing}
.......\=.......\=.......\=.......\=.......\=      \kill
\> {\bf for} every box $j$ on level $l_{cut}$ \\
\> \>  $\bullet$ Set $\Psi_{l_{cut},j} = 0$ \\
\> {\bf end} \\
\> {\bf for} $l = l_{cut}+1,\dots,l_{max}$  \\
\> \>  {\bf for} every box $j$ on level $l$: \\
\> \> \> $\bullet$ Compute $\tilde \Psi_{l,j}$ from its parent's $\Psi$ 
expansion using the operator \\
\> \> \> $\ \ $ ${\cal T}_{LL}$ \\
\> \> \> {\bf for} every box $m$ in $j$'s interaction list: \\
\> \> \> \>  $\bullet$ Increment $\Psi_{l,j}$ by adding in the 
contributions from all boxes \\ 
\> \> \> \> $\ \ $ in $j$'s interaction list, using (\ref{lotanew}).\\
\> \> \> {\bf if} $j$ is childless {\bf then} \\
\> \> \> \> {\bf for} every box $m$ in $j$'s {\em s-list:} \\
\> \> \> \> \> $\bullet$ Evaluate the Hermite expansion $\Phi(m)$ at each
target \\
\> \> \> \> \> $\ \ $ in box $j$. \\
\> \> \> \>  {\bf end} \\
\> \> \> \> {\bf for} every box $m$ in $j$'s {\em s-list:} \\
\> \> \> \> \> $\bullet$ Increment the local expansion $\Psi(m)$ from the
smooth \\
\> \> \> \> \> $\ \ $ and point source distribution in $j$, using the precomputed 
\\
\> \> \> \> \> $\ \ $ operators \eqref{tdirect} for the smooth source distribution  \\
\> \> \> \> \> $\ \ $and \eqref{locexp2d} for the point sources \\
\> \> \> \>  {\bf end} \\
\> \> \> \> $\bullet$ Evaluate the local expansion $\Psi_{l,j}$ at each 
target in box $j$ \\
\> \> \> \> $\ \ $ (whether the target is a grid point or a point
source location) \\
\> \> \> {\bf endif} \\
\> \> {\bf end} \\ 
\> {\bf end} 
\end{tabbing}

              \step{Step III: Local interactions}

\begin{tabbing}
.......\=.......\=.......\=.......\=      \kill
\> {\bf for} $l = 0,\dots,l_{max}$  \\
\> \>  {\bf for} every {\em leaf node} $B_j$ on level $l$: \\
\> \> \> $\bullet$ At each tensor product grid point in $B_j$, 
compute influence of the \\
\> \> \> $\ \ $ smooth source in colleagues, fine neighbors and coarse \\
\> \> \> $\ \ $ neighbors using precomputed tables of coefficients 
\eqref{eq:localweights} \\
\> \> \> $\bullet$ For each point source location in $B_j$, use Chebyshev 
interpolation \\
\> \> \> $\ \ $ to obtain the Gaussian field due to smooth sources in 
colleagues, \\
\> \> \> $\ \ $ fine and coarse neighbors \\
\> \> \> $\bullet$ For {\em all} targets in $B_j$, use direct computation to 
evaluate  the \\
\> \> \> $\ \ $ Gaussian field due to point sources in colleagues, \\
\> \> \> $\ \ $ fine and coarse neighbors \\
\> \>  {\bf end} \\
\> {\bf end}

\end{tabbing}

\vspace{.2in}
\noindent

The cost of the adaptive FGT is easily estimated.
Creating the tree and sorting sources into leaf nodes requires
at most $O((N + N_s) l_{max})$ work.
Forming expansions on all leaf nodes requires
$O( (N+N_s) p^2$ work, for an expansion of order $p$.
The remainder of the upward pass requires
$O(N_b p^3)$ work to carry out the recursive merging of Hermite
expansions, where $N_b$ is the number of boxes in the quad-tree.
The downward pass requires approximately
$O(27 N_b p^3)$ work to carry out the Hermite-local and local-local
translations.
Finally, the local work is of the order 
$O(N_s q)$ for the point sources (assuming the tree has been refined 
until there are $O(q)$ sources per leaf node). For the continuous source
distribution, 
only approximately $13 N \, \frac{k(k+1)}{2} + N_s k^2$ operations are required.
The first term accounts for the cost of computing the Gaussian potential
on the tensor product grids from the near neighbors, 
using precomputed tables, while the latter term
is the interpolation cost at the 
point source locations.
The factor $13$ is a consequence of the observation that the
maximum number of neighbors a box can have
is thirteen (twelve fine neighbors and itself). 

\begin{remark}
The preceding analysis assumes that the translation operators
${\cal T}_{HH}$, ${\cal T}_{LL}$, and ${\cal T}_{HL}$
have been computed according to the formulae
\eqref{lolo}, \eqref{tata} and \eqref{lota}, taking
advantage of the tensor product nature of the two-dimensional
Hermite and local expansions to achieve $O(p^3)$ complexity, instead 
of the naive estimate $O(p^4)$. In the $d$-dimensional setting,
the operation count is $O(d p^{d+1})$ instead of $O(p^{2d})$ 
\cite{greengard_strain1,greengard98}.
\end{remark}

We have further accelerated the code by making use of 
diagonal translation operators, following the method described in 
\cite{greengard98} and \cite{fgt3}. Instead of Hermite expansions,
it is straightforward to show that 
\begin{equation}
\sum_{\alpha\geq 0}A_{\alpha} h_{\alpha}\left(\frac{\x-\s_B}{\sqrt{\delta}}
\right) =
\int_{\mathbb{R}^2}
w(\k) e^{-\frac{\|\k\|^2}{4}} 
e^{i \k \cdot (\x - \s^B)/\sqrt{\delta}}  \,d\k
\, ,
\label{pwformula}
\end{equation}
where 
\begin{equation}
w(\k) = w(k_1,k_2) = \sum_{\alpha\geq 0} A_{\alpha} (-i)^{|\alpha|} 
k_1^{\alpha_1} k_{2}^{\alpha_2}. 
\end{equation}
This formula is derived from the Fourier relation
\begin{equation}
e^{-\|\x\|^2} = \left(\frac{1}{4\pi}\right)
\int_{\mathbb{R}^2} e^{-\|\k\|^2}{4}e^{i\k \cdot \x}\, d\k. 
\end{equation}
In order to make practical use of 
\eqref{pwformula}, we need to discretize the integral, for which the 
trapezoidal rule is extremely efficient because of the smoothness
and exponential decay of the integrand.
The reason \eqref{pwformula} is useful is that it provides a 
basis in which translation is diagonal. Assuming $p_t$ denotes the 
number of trapezoidal quadrature points required, it is shown in
\cite{greengard98} and \cite{fgt3} that the dominant cost 
of translating Hermite to local expansions, namely the $O(27 N_b p^3)$
term above, can be reduced to $O(3 N_b p_t^2 + N_b p p_t^2)$ work.

The principal difference 
between the methods in  
\cite{greengard98} and \cite{fgt3} and our
hierarchical scheme is that
$p_t$ must be different on each level. Informally speaking, for 
a level where the linear box size is $r_l$, $p_t$ must be 
sufficiently large that the integrand
$e^{i \k \cdot (\x - \s^B)/\sqrt{\delta}}$ is Nyquist-sampled for
$(\x - \s^B) \leq 4 r_l$, where $r_l < r_c \sqrt{\delta}$
and $r_c$ is the cutoff parameter defined above.
(It is easy to verify that
$p_t = O(p)$  \cite{greengard98,fgt3}.) 

\subsection{Boundary FGT} \label{bdryfgt}

We turn now to the evaluation of boundary Gauss transforms
of the form  \eqref{eq:bdryint}, 
for targets both on and off the boundary $\Gamma$. 
We assume that $\Gamma$ itself 
is described as the union of $M_b$ boundary segments:
\[ \Gamma = \cup_{j=1}^{M_b} \Gamma_j, \]
with each boundary segment defined by a $k$th order Legendre
series. That is,
\begin{equation}
 \Gamma_j = \Gamma_j(s) = (x^1_j(s),x^2_j(s)): \quad
 x^1_j(s) = \sum_{n=0}^{k-1}  x^1_j(n) P_n(s), \ 
x^2_j(s) = \sum_{n=0}^{k-1}  x^2_j(n) P_n(s),  
\label{bdrydiscrete}
\end{equation}
with $-1 \leq s \leq 1$.
We assume that the densities $\sigma$ and $\mu$ in \eqref{eq:bdryint}
are also given by corresponding piecewise Legendre series:
\[ \sigma_j(s) = \sum_{n=0}^{k-1}  \sigma_j(n) P_n(s), \quad
\mu_j(s) = \sum_{n=0}^{k-1} \mu_j(n) P_n(s).  \]

For the sake of simplicity, we assume
that $\Gamma$ has been discretized in a manner that is 
commensurate with the underlying data structure used above:
an adaptive, level-restricted tree.
That is, we assume the length of $\Gamma_j$,
denoted by $| \Gamma_j|$ satisfies $|\Gamma_j| \approx r_l$, 
where $r_l$ is the box size of the leaf node in the tree that contains 
the center point $\c_j$ of $\Gamma_j$.

\begin{figure}[htbp]
\centering
\includegraphics[width=.4\textwidth]{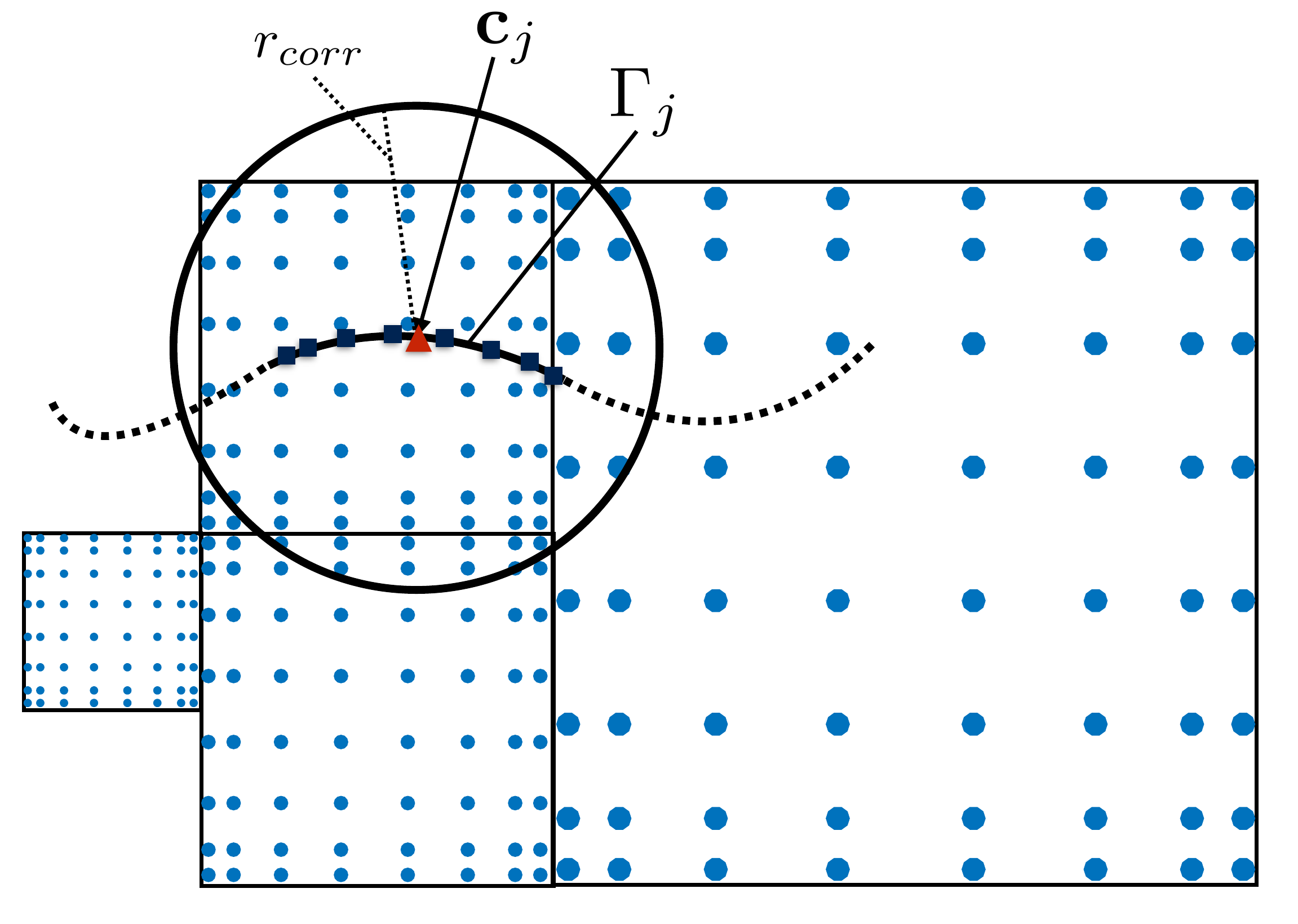} \hspace{.3in}
\includegraphics[width=.25\textwidth]{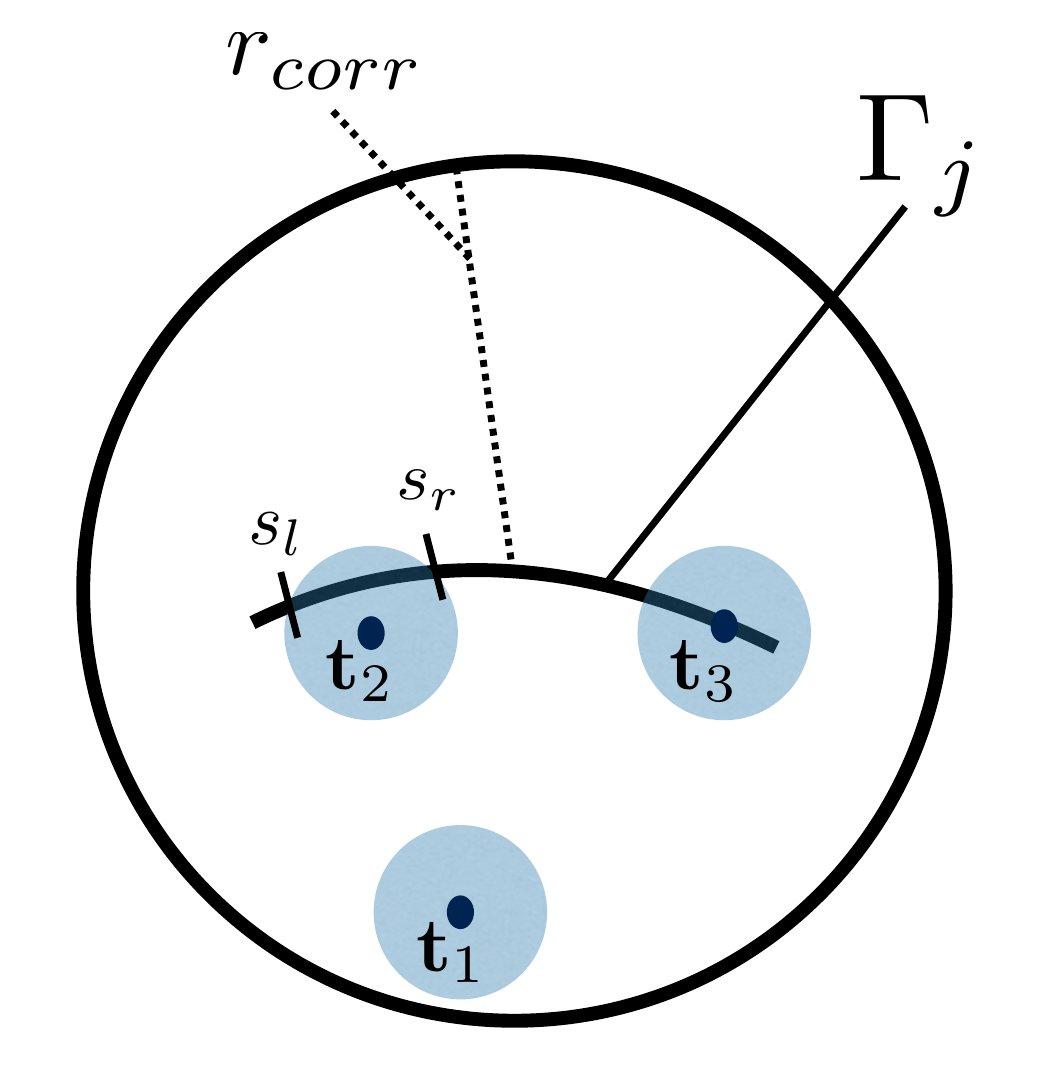}
\caption{(Left) A boundary segment $\Gamma_j$ wth center point $\c_j$ 
lying in a leaf node $B$ of side length $r_l$. 
Depending on the value of $\delta$, 
a boundary integral of the form  \eqref{eq:bdryint}
is either resolved by its discretization using standard Gauss-Legendre
quadrature with $k$ nodes on $\Gamma_j$, or negligible outside the disk
centered at $\c_j$ with radius $r_{corr}$, which we will denote by $D_j$.
(Right) In the latter case (when 
$\delta$ is small and the Gaussian is sharply peaked), 
a simple interpolatory rule can be used to 
compute the correct contribution using $O(k^2)$ work per target point,
either on or off the boundary. The shaded circles in the figure around
the three target points $\t_i$ are meant to indicate the regions where 
the Gaussians centered at $\t_i$ are less than a user-prescribed tolerance
$\epsilon$. 
$\t_1$ is sufficiently far from $\Gamma_j$ that it can be
ignored. $\t_2$ and $\t_3$ are off and on the boundary, respectively.
The relevant portion of $\Gamma_j$ for $\t_2$ is marked in terms of 
the parameter $s$ by $s_l$ and $s_r$.}
\label{bfgtfig}
\end{figure}

Suppose now that we apply composite Gauss-Legendre quadrature to the 
integrals in \eqref{eq:bdryint}. For the ``single layer" type integral,
we have
\begin{equation}
 {\cal S}[\sigma](\x) =
\int_{\Gamma} e^{-\frac{|\x-\y(s)|^2}{\delta}} \sigma(\y(s)) \, ds_{\y} \, 
\approx
\sum_{j=1}^{M_b} \sum_{i=0}^{k-1}
 e^{-\frac{|\x-\y_{ij}|^2}{\delta}} \sigma_{ij} w_{ij}  \, ,
\label{eq:bdrydisc}
\end{equation}
where $\y_{ij} = (y^1_{ij},y^2_{ij})$ is the location of
the $i$th scaled Gauss-Legendre node on
$\Gamma_j$, $\sigma_{ij}$ is the density value at that point,
and $w_{ij} = w_i \sqrt{ [dy^1_j/ds(s_i)]^2 + 
[dy^2_j/ds(s_i)]^2}$.
Here, $s_i$ and $w_i$ denote the standard Gauss-Legendre nodes and weights
on $[-1,1]$. 

Note that the quadrature weight $w_{ij}$ involves 
both the standard weight $w_i$ and the change of variables 
corresponding to an arc-length parametrization on each segment. 
The necessary derivatives can be computed from \eqref{bdrydiscrete}.
Note also that the sum in \eqref{eq:bdrydisc} consists simply of
point sources and is easily incorporated into the FGT above.
The only remaining issue has to do with the accuracy of the formula
\eqref{eq:bdrydisc}, since the smoothness of the integrand depends
strongly on the parameter $\delta$. Here, however, the rapid decay of the
Gaussian makes the problem tractable for any $\delta$.
To see why, consider a boundary segment $\Gamma_j$, centered at $\c_j$ 
in a leaf box $B$ of commensurate size (Fig. \ref{bfgtfig}, left). 
Suppose first that $\delta$ is sufficiently large that
$|\Gamma_j| \leq C(\epsilon) \sqrt{\delta}$. 
$C(\epsilon)$ is straightforward to tabulate in terms of the 
user-specified tolerance $\epsilon$.
Then, it is straightforward to 
show that the error in $k$-point Gauss-Legendre quadrature is of the order 
\[ E_k \approx C' \left( \frac{e |\Gamma_j|^2}{16 \delta k} \right)^k    \]
for some constant $C'$.
This follows from the error estimate for Gauss-Legendre quadrature
\cite{DR},
\[  E \leq \frac{(b-a)^{2k+1} (k!)^4}{(2k+1)[(2k)!]^3} 
\| f^{(2k)}\|_\infty
\]
for the integral $\int_a^b f(x) \, dx$, Stirling's formula,
and Cramer's inequality. In short, the $k$-point quadrature rule is 
{\em spectrally accurate}. 
Suppose now that
$|\Gamma_j| > C(\epsilon) \sqrt{\delta}$, and let $r_{corr} = |\Gamma_j|$.
Then, for any target outside the circle of radius $r_{corr}$ centered
at $\c_j$, the integrand is bounded by 
$e^{-C(\epsilon)^2/4} \, |\Gamma_j| \, \| \sigma \|_\infty$. 
(Setting $C = 12$, the integrand is approximately 
$|\Gamma_j| \, \| \sigma \|_\infty \, (2 \, \cdot 10^{-16})$.) 
Thus, for each boundary segment with 
$|\Gamma_j| > C(\epsilon) \sqrt{\delta}$,
it remains only to correct the result obtained from the FGT 
within this circle $D_j$ (Fig. \ref{bfgtfig}, right). 

This correction can be computed rapidly and accurately as follows.
Let us denote by $D(\t,r)$ the circle of radius $r$ centered at $\t$ 
for any target $\t$ in the circle, where 
$r = \sqrt{\delta \ln(1/\epsilon)}$ so that
$e^{-\|\t-\y\|^2/\delta} < \epsilon$. If $D(\t,r)$ doesn't intersect
$\Gamma_j$, the field is negligible and no correction is needed.
Otherwise, we compute the intersection of 
$\Gamma_j$ and $D(\t,r)$ and let the endpoints of the intersection be denoted 
by $s_l$ and $s_r$ (in terms of the underlying parametrization of $\Gamma_j$).
We then interpolate the source distribution $\sigma(s)$
to $k_c$ scaled Gauss-Legendre nodes on $\Gamma_j^{s}$ for 
$s \in [s_l,s_r]$, and
replace the original $k$-point quadrature on $\Gamma_j$ 
by a $k_c$-point Gauss-Legendre rule on $[s_l,s_r]$.
Setting $k_c$ to 20 yields approximately fourteen digits of accuracy
assuming the density $\sigma(s)$ is locally smooth.

\subsection{Periodic boundary conditions}

It is straightforward to extend the FGT to handle periodic conditions
on the unit square $D = [-0.5,0.5]^2$. Conceptually speaking,
this can be accomplished by tiling the entire plane
${\bf R}^2$ with copies of the source distribution $f$.
For this, we let $\Lambda = \{ \j = (j_1,j_2) | j_1,j_2 \in {\bf Z} \}$. 
The tile ${\rm T}_{\j}$ 
is a unit square centered at the lattice point 
$\j  \in \Lambda$. The extended periodic source distribution 
will be denoted by $\tilde{f}$. From this, the solution to the periodic
problem can be written as:
\[ 
\tilde{F}(\x)=\int_{{\bf R}^2} e^{-\frac{|\x-\y|^2}{\delta}}\tilde{f}(\y) \, d\y.\;\;\;\x\in D.
\]

As in the FMM for the Poisson equation \cite{MLFMM,Ethridge2001},
we can accomodate periodic boundary conditions with very little 
change to the data structure or processing.
To see this, note that, if we carry out the upward pass of the FGT until the 
root node (level $l=0$), we obtain an Hermite expansion describing the
field due to all sources in $D$. Because of the translation invariance
of the kernel, the coefficients of this expansion are
the same for every tile ${\rm T}_{\j}$ covering the plane, expanded
about the corresponding lattice point $\j$. 
We denote the expansion about $\j$ by
\[ 
\phi(\x)=\sum_{\alpha\leq p} A_{\alpha}
h_{\alpha}\left(\frac{\x - \j}{\sqrt{\delta}}\right).
\]
Let us now define the punctured lattice by 
$$\Lambda' = \Lambda - \{
(-1,-1), (-1,0), (-1,1), (0,-1), (0,1), (1,-1), (1,0), (1,1) \}.$$
Having deleted the colleagues of the original unit box $D$, centered at the
origin, the remaining tiles indexed by $\Lambda'$ are all well separated
from $D$. 
From Lemma \ref{lemma:H2L} and the linearity of the problem, it is 
clear that the 
contribution to field in $D$ from {\em all} well-separated tiles
indexed by $\Lambda'$ can be representation by a local expansion 
\begin{equation}
\tilde{F}_{far}(\x)=
\sum_{\beta\geq 0} C_{\beta}\left(\frac{\x}{\sqrt{\delta}}\right)^{\beta},
\label{psi01}
\end{equation}
centered at the origin, with coefficients
\begin{equation}
C_{\beta}=\frac{(-1)^{|\beta|}}{\beta!}\sum_{\alpha\leq p} A_{\alpha}
{\cal L}_{\alpha+\beta}
   \label{eq:imagall}
\end{equation}
where
\begin{equation}
{\cal L}_{\alpha+\beta} =
\sum_{\j\in\Lambda'}h_{\alpha+\beta}\left(\frac{\j}{\sqrt{\delta}}\right).
   \label{latsums}
\end{equation}

Extending Definitions \ref{ldef}-\ref{def:exps}, we
let $\Psi_{0,1}$ denote the local expansion for the root node $D$ at level $0$,
and define the nine tiles centered at $\Lambda - \Lambda'$ to be the 
root node's colleagues.

When $\delta$ is so small that $l_{cut}\geq 0$, the far field 
$\tilde{F}_{far}(\x)$ in the root node $D$
is negligible (for given accuracy $\epsilon$), so that we can initialize
its coefficients $C_{\beta} = 0$. Otherwise, we carry out the 
computation in \eqref{eq:imagall} to initialize $\Psi_{1,0}$.
This requires the computation of tte lattice sums in 
\eqref{latsums}.
These are obtained rapidly from the Poisson summation formula:
\begin{equation} \label{eq:poissum}
\begin{split}
\sum_{j_1=-\infty}^{\infty} 
\sum_{j_2=-\infty}^{\infty} 
&h_{\alpha}\left(\frac{(j_1,j_2)}{\sqrt{\delta}}\right)= \\
&\pi^2\delta^2 (-2\pi i \sqrt{\delta})^{\alpha_1+\alpha_2}
\left( 
\sum_{m=-\infty}^{\infty} m^{\alpha_1}\cdot e^{-\pi^2m^2\delta} \right)
\left(
\sum_{n=-\infty}^{\infty} n^{\alpha_2}\cdot e^{-\pi^2n^2\delta} \right).
\end{split}
\end{equation}

It is straightforward to verify that, when $\delta$ is large enough that
$\tilde{F}_{far}(\x)$ is non-negligible, only a few terms are required on
the right-hand side of \eqref{eq:poissum} and only milliseconds are needed
for all ${\cal L}_{\alpha+\beta}$.
 
Only two other changes are needed in the FGT:
the interaction list and near neighbor computations must be adjusted 
to account for periodic images. Having defined the colleagues of 
the root node above, this is handled automatically by the data structure.
For large scale problems with many levels of refinement, this involves
a modest increase in work for boxes near the boudary of $D$ and a negligible
increase in the total work.

\section{Numerical results} \label{sec_results}
In this section, we illustrate the performance of both the
volume and boundary FGT, implemented in Fortran, with 
experiments carried out on a single core of a 3.4GHz Intel Xeon processor.
Our first example demonstrates the linear
scaling of CPU time with the number of grid points. We compute
the volume FGT
\begin{equation}
G_{\delta}[f_k](\x)=\int_{{\bf R}^2} e^{-\frac{|\x-\y|^2}{\delta}} f(\y) d\y \, ,
\end{equation}
with the source distribution
\begin{equation} \label{eq:exper2}
f_k(\x)=\sin(2k\pi x_1) \cos(2k\pi x_2) \;\;\; (k\in{\bf Z}),
\end{equation}
imposing periodic boundary conditions.

In order for the numerical experiment to be non-trivial, we increase
the complexity of the problem as we increase the number of degrees of freedom.
More precisely, we consider four cases, 
with $k=1, 2, 4, 8$ and $\delta=\frac{1}{k^2}$, 
requiring a finer and finer spatial mesh to resolve the data.
For each choice of $k$, we create a (level-restricted) quadtree, refined to 
a level where $f_k$ is accurately represented with our piecewise
eighth order polynomial to 10 digits of accuracy.
For the function described in \eqref{eq:exper2}, 
the refinement happens to be uniform, with
$N=256, 1024, 4096, 16384$ leaf nodes for the four cases, respectively. 
(We will see examples with inhomgeneous source distributions and 
adaptive data structures below.)
Timings are given in
Table \ref{fgtvoltable3} and plotted in Fig. \ref{fgtvolfig3}.

\begin{table}[htbp]
\begin{center}
 \begin{tabular}{|c|c|c|c|c|}
 \hline
   & \multicolumn{4}{|c|}{$k$} \\ \hline
  $\epsilon$ & $1$ & $2$ & $4$ & $8$ \\ \hline
  $10^{-3}$ & $3.0 \cdot 10^5$ & $3.3 \cdot 10^5$ & $6.1 \cdot 10^5$ & $7.3 \cdot 10^5$ \\ \hline
  $10^{-6}$ & $1.7 \cdot 10^5$ & $1.8 \cdot 10^5$ & $3.1 \cdot 10^5$ & $4.3 \cdot 10^5$ \\ \hline
  $10^{-9}$ & $0.9 \cdot 10^5$ & $0.9 \cdot 10^5$ & $1.1 \cdot 10^5$ & $1.3 \cdot 10^5$ \\ \hline
 \end{tabular}
 \end{center}
 \caption{Throughput on a single core for the volume FGT with periodic boundary condition in units of points/second.}
\label{fgtvoltable3}
\end{table}

\begin{figure}[htbp]
\centering
\includegraphics[width=.55\textwidth]{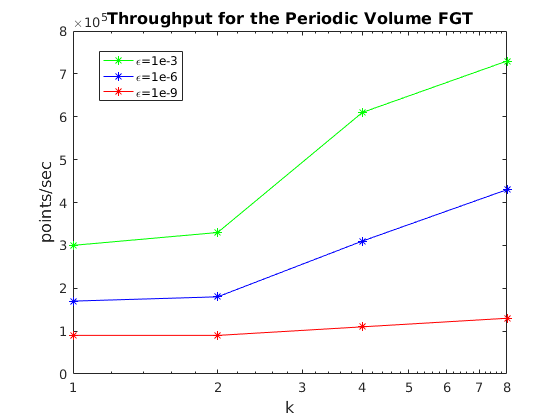}
\caption{Throughput for the volume FGT 
(measured in units of $100,000$ points/second) with various precisions,
plotted as a function of $\delta$. (The data is the same as in Table 
\ref{fgtvoltable3}.)}
\label{fgtvolfig3}
\end{figure}

While the cost appears to grow in a sublinear fashion with the number of 
grid points, this is simply because of the non-trivial cost of the precomputation.
A more precise model for the CPU time 
takes the form 
\begin{equation}
T(N,\epsilon)=A(\epsilon)N+B(\epsilon) \log N 
\end{equation}
for precision $\epsilon$.
The term $B(\epsilon) \log N$ is dominated by the building of tables 
for the local interactions, which is done once per level.
In the present example, the sublinear
part contributes about 30 \% of the cost for the smaller problem sizes and less than
10 \% for the largest $N$.
If we subtract the time for precomputation/table building,
and measure the time of the remainder of the FGT, 
we see a steady ``throughput," measured in {\em points per second} 
for each fixed precision. This verifies the linear scaling
(Table \ref{fgtvoltable3-2} and Fig. \ref{fgtvolfig3-2}).

\begin{table}[htbp]
\begin{center}
 \begin{tabular}{|c|c|c|c|c|}
 \hline
   & \multicolumn{4}{|c|}{$k$} \\ \hline
  $\epsilon$ & $1$ & $2$ & $4$ & $8$ \\ \hline
  $10^{-3}$ & $7.0 \cdot 10^5$ & $7.2 \cdot 10^5$ & $8.2 \cdot 10^5$ & $8.1 \cdot 10^5$ \\ \hline
  $10^{-6}$ & $3.8 \cdot 10^5$ & $4.0 \cdot 10^5$ & $4.3 \cdot 10^5$ & $4.1 \cdot 10^5$ \\ \hline
  $10^{-9}$ & $1.5 \cdot 10^5$ & $1.5 \cdot 10^5$ & $1.4 \cdot 10^5$ & $1.5 \cdot 10^5$ \\ \hline
 \end{tabular}
 \end{center}
 \caption{Throughput on a single core for the volume FGT with periodic boundary condition in units of points/second.}
\label{fgtvoltable3-2}
\end{table}

\begin{figure}[htbp]
\centering
\includegraphics[width=.55\textwidth]{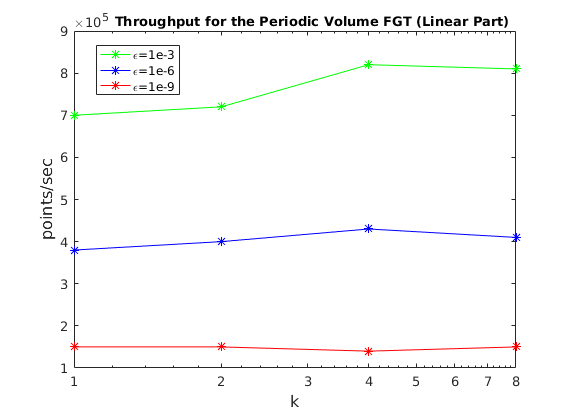}
\caption{Throughput for the volume FGT 
(measured in units of $100,000$ points/second) with various precisions,
plotted as a function of $\delta$. (The data is the same as in Table 
\ref{fgtvoltable3-2}.)}
\label{fgtvolfig3-2}
\end{figure}

For our second example, we again compute the volume FGT with periodic 
boundary condition
where $f_k$ is given by \eqref{eq:exper2} with $k = 2$.
$G_{\delta}[f_k]$ is available analytically for this $f_k(\x)$ from Fourier 
analysis.
The source distribution is again resolved to 10 digits of accuracy, but we now
compute the FGT with requested precisions of
$\epsilon=10^{-3}$, $10^{-6}$ and $10^{-9}$. For each choice of $\epsilon$, 
we carry out the computation for a wide range of $\delta$,
from $\delta=10^{-7}$ to $\delta=10^{-1}$.
Timings are provided in Table \ref{fgtvoltable2} and plotted in
Fig. \ref{fgtvolfig2}. 

\begin{table}[htbp]
\begin{center}
 \begin{tabular}{|c|c|c|c|c|c|c|c|}
 \hline
   & \multicolumn{7}{|c|}{$\delta$} \\ \hline
  $\epsilon$ & $10^{-1}$ & $10^{-2}$ & $10^{-3}$ & $10^{-4}$ & $10^{-5}$ & $10^{-6}$ & $10^{-7}$\\ \hline
  $10^{-3}$ & $3.4 \cdot 10^5$ & $3.5 \cdot 10^5$ & $3.3 \cdot 10^5$ & $5.5 \cdot 10^5$ & $6.2 \cdot 10^5$ & $6.9 \cdot 10^5$ & $7.3 \cdot 10^5$\\ \hline
  $10^{-6}$ & $1.7 \cdot 10^5$ & $1.7 \cdot 10^5$ & $1.7 \cdot 10^5$ & $0.7 \cdot 10^5$ & $3.0 \cdot 10^5$ & $5.0 \cdot 10^5$ & $5.4 \cdot 10^5$\\ \hline
  $10^{-9}$ & $0.8 \cdot 10^5$ & $0.9 \cdot 10^5$ & $0.6 \cdot 10^5$ & $0.3 \cdot 10^5$ & $3.5 \cdot 10^5$ & $4.0 \cdot 10^5$ & $4.4 \cdot 10^5$\\ \hline
 \end{tabular}
 \end{center}
 \caption{Throughput on a single core for the volume FGT with periodic boundary condition in units of points/second.}
\label{fgtvoltable2}
\end{table}

\begin{figure}[htbp]
\centering
\includegraphics[width=.55\textwidth]{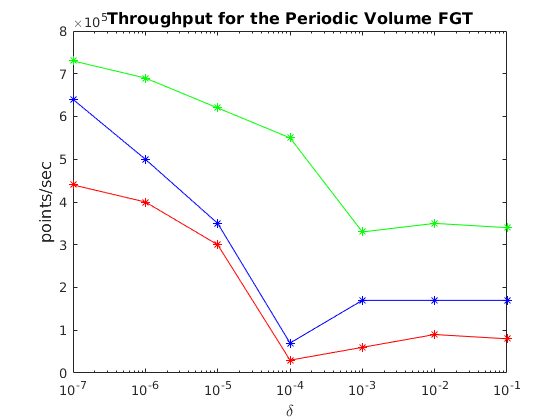}
\caption{Throughput for the volume FGT 
(measured in units of $100,000$ points/second) with various precisions,
plotted as a function of $\delta$. (The data is the same as in Table 
\ref{fgtvoltable2}.)}
\label{fgtvolfig2}
\end{figure}

Our third example illustrates the performance of the FGT with an adaptive data
structure. For this, we let $B$ denote the unit box with $f$ given by
\begin{equation}
f(\x)=\sum_{i=1}^5 e^{-\alpha_i|\x-\x_i|^2} \, ,
\end{equation}
with
\[ [\x_1,\dots,\x_5] = [(0.20,0.10), (0.31,0.50),
 (0.68,0.40), (0.41,0.80), (0.12,0.45)] \]
and
\[ [\alpha_1,\dots,\alpha_5]  = (0.010,0.005,0.003,0.002,0.001). \]

\begin{figure}[htbp]
\centering
\includegraphics[width=.45\textwidth]{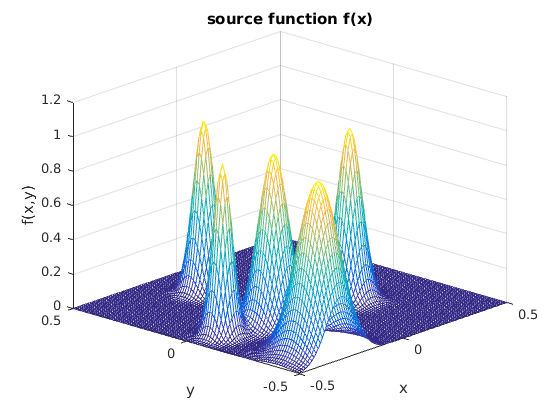} \hspace{.3in}
\includegraphics[width=.45\textwidth]{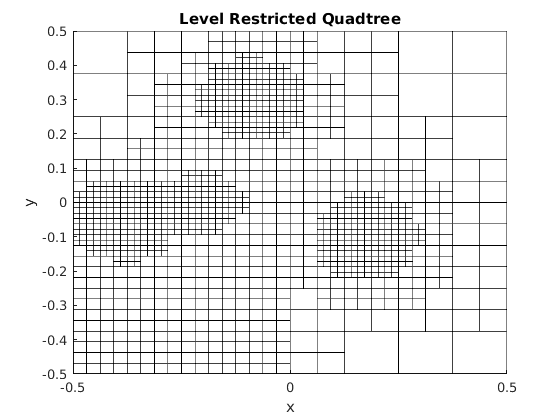}
\caption{(Left) Plot of the volume source distribution, which is taken to be the sum of a few Gaussians. (Right) A level restricted quadtree on which the source distribution is resolved to 10 digits of accuracy as a piecewise
polynomial of degree 8.}
\label{fgtvolsource}
\end{figure}

$G_{\delta}[f]$ is available analytically for this choice of $f(\x)$. 
In Fig. \ref{fgtvolsource}, we plot the source distribution along
with a level restricted quadtree on which the source distribution is resolved
to ten digits of accuracy.
We compute the volume FGT with requested precisions of
$\epsilon=10^{-3}$, $10^{-6}$ and $10^{-9}$. For each choice of $\epsilon$, 
we carry out the computation for a wide range of $\delta$,
from $\delta=10^{-7}$ to $\delta=10^{-1}$. 
Timings are given in Table \ref{fgtvoltable} and plotted in
Fig. \ref{fgtvolfig}.

\begin{table}[htbp]
\begin{center}
 \begin{tabular}{|c|c|c|c|c|c|c|c|}
 \hline
   & \multicolumn{7}{|c|}{$\delta$} \\ \hline
  $\epsilon$ & $10^{-1}$ & $10^{-2}$ & $10^{-3}$ & $10^{-4}$ & $10^{-5}$ & $10^{-6}$ & $10^{-7}$\\ \hline
  $10^{-3}$ & $2.8 \cdot 10^5$ & $2.8 \cdot 10^5$ & $3.4 \cdot 10^5$ & $4.3 \cdot 10^5$ & $5.0 \cdot 10^5$ & $7.5 \cdot 10^5$ & $8.0 \cdot 10^5$\\ \hline
  $10^{-6}$ & $1.6 \cdot 10^5$ & $1.6 \cdot 10^5$ & $1.5 \cdot 10^5$ & $0.9 \cdot 10^5$ & $3.1 \cdot 10^5$ & $5.4 \cdot 10^5$ & $6.0 \cdot 10^5$\\ \hline
  $10^{-9}$ & $0.8 \cdot 10^5$ & $0.8 \cdot 10^5$ & $0.5 \cdot 10^5$ & $0.4 \cdot 10^5$ & $3.6 \cdot 10^5$ & $4.3 \cdot 10^5$ & $4.8 \cdot 10^5$\\ \hline
 \end{tabular}
 \end{center}
 \caption{Throughput on a single core for the volume FGT in units of
points/second.  This is a useful benchmark for linear scaling algorithms,
permitting simple estimation of the performance in terms of CPU time to 
any problem size.
}
\label{fgtvoltable}
\end{table}

\begin{figure}[htbp]
\centering
\includegraphics[width=.55\textwidth]{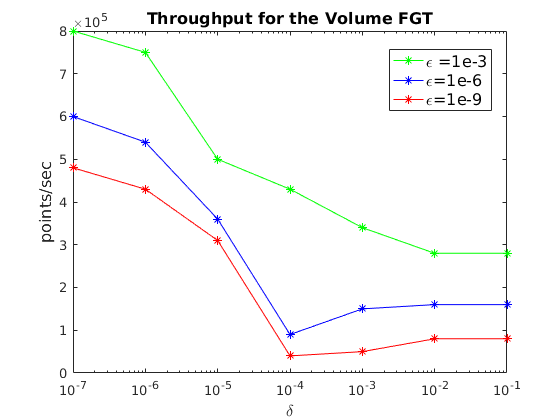}
\caption{Throughput for the volume FGT 
(measured in units of $100,000$ points/second) with various precisions,
plotted as a function of $\delta$. (The data is the same as in Table 
\ref{fgtvoltable}.)}
\label{fgtvolfig}
\end{figure}

Note that, for a fixed tolerance $\epsilon$, the performance of the
volume FGT is relatively insensitive to the parameter
$\delta$. For large $\delta$, the far field is nontrivial but very smooth.
For sufficiently small $\delta$,
the interaction is entirely local and the FGT is particularly fast.
The worst case performance is for 
$\delta \approx 10^{-4}$, where both the far field and near field
need significant effort. 

Our fourth example illustrates the performance of the boundary FGT. 
We compute the integral
\begin{equation}
G_{\delta}[f](\x)=\int_{\Gamma} e^{-\frac{|\x-\y|^2}{\delta}} f(\y) d\y \, ,
\end{equation}
where $\Gamma$ is chosen to be the ellipse:
\begin{equation}
\begin{cases}
y_1(\theta) = 0.45 \cos(\theta), \\
y_2(\theta) =0.25 \sin(\theta), 
\end{cases}
(0 \leq \theta \leq 2\pi) \, .
\end{equation}
We let 
\begin{equation}
f(\x) = \cos(2x_1)+\sin(x_2),\;\;\;(\x\in\Gamma).
\end{equation}

We create an adaptive quadtree on the unit box so that each leaf box of the
tree contains no more than $O(1)$ boundary points and then enforce the 
level-restricted condition, yielding the data structure shown in 
Fig. \ref{bdryfig}. The leaf nodes with $8 \time 8$ tensor product Chebyshev
grids on each define our volumetric targets. 
The boundary FGT is then evaluated at all volumetric grid points and all boundary
points as well. 
Timings are given in Table \ref{fgtbdrytable} and plotted in
Fig. \ref{fgtbdryfig}.

\begin{figure}[htbp]
\centering
\includegraphics[width=.55\textwidth]{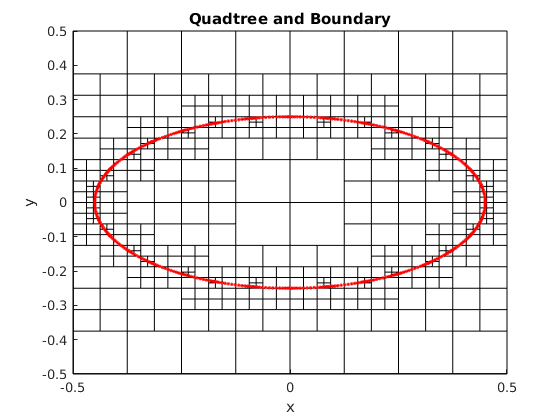}
\caption{A level-restricted quad tree determined by the 
discretization of the indicated ellipse. 
} 
\label{bdryfig}
\end{figure}

\begin{table}[htbp]
\begin{center}
 \begin{tabular}{|c|c|c|c|c|c|c|c|}
 \hline
   & \multicolumn{7}{|c|}{$\delta$} \\ \hline
  $\epsilon$ & $10^0$& $10^{-1}$ & $10^{-2}$ & $10^{-3}$ & $10^{-4}$ & $10^{-5}$ & $10^{-6}$\\ \hline
  $10^{-3}$ & $4.2\cdot 10^5$ & $3.9\cdot 10^5$ & $4.1\cdot 10^5$ & $3.6\cdot 10^5$ & $5.0\cdot 10^5$ & $2.8\cdot 10^5$ & $5.4\cdot 10^5$ \\ \hline
  $10^{-6}$ & $2.0\cdot 10^5$ & $1.8\cdot 10^5$ & $1.6\cdot 10^5$ & $1.6\cdot 10^5$ & $0.9\cdot 10^5$ & $1.8\cdot 10^5$ & $3.9\cdot 10^5$ \\ \hline
  $10^{-9}$ & $1.0\cdot 10^5$ & $0.9\cdot 10^5$ & $0.8\cdot 10^5$ & $0.4\cdot 10^5$ & $0.2\cdot 10^5$ & $1.1\cdot 10^5$ & $2.7\cdot 10^5$ \\ \hline
 \end{tabular}
 \end{center}
\caption{Throughput for the boundary FGT with volumetric targets,
measured in points/second, for various precisions
and values of $\delta$.} 
 \label{fgtbdrytable}
\end{table}

\begin{figure}[htbp]
\centering
\includegraphics[width=.55\textwidth]{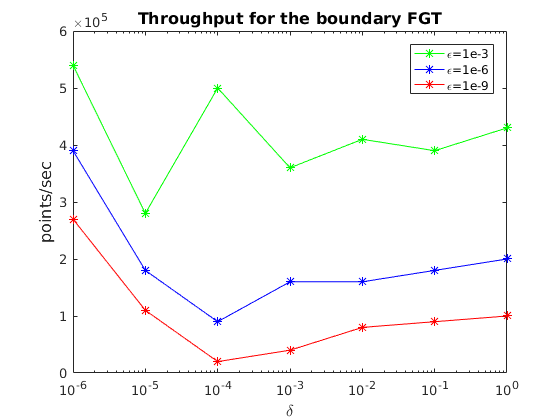}
\caption{Throughput for the boundary FGT with boundary and volume targets,
measured in points/second, for various precisions,
plotted as a function of $\delta$.} 
\label{fgtbdryfig}
\end{figure}

Note, again, that the performance of the boundary FGT varies only modestly over
a wide range of the parameter $\delta$. For sufficiently small $\delta$,
the interaction is entirely local and
no expansions are formed. For sufficiently large $\delta$, a smooth quadrature 
rule is accurate enough to discretize
the boundary integral, avoiding the need for local correction. 
The code is slowest for intermediate values of $\delta$, where both local
and far field contributions are significant (while still satisfying linear scaling with 
the number of source and target points).

\subsection{An initial value problem for the heat equation}

As a final example, we consider the homogeneous heat equation:
\begin{equation} \label{heateqper}
\begin{split}
u_{t}(\x,t) &= \Delta u(\x,t), \\
u(\x,0) &= f(\x)
\end{split}
\end{equation}
for $\x\in D = [-0.5,0.5]^2$, with 
periodic boundary conditions. The initial data is chosen to be
a piecewise constant function:
\begin{equation}
f(\x)=C_i,\;\;\; {\rm for} \;\;\x\in D_i,
\end{equation}
where the unit box $D$ is refined uniformly on a tree that is five levels
deep, resulting in a 32-by-32 grid of leaf nodes $D_i$.
On each leaf node, we let $C_i$ take on a random value in the range
$[0,1]$. $f(x)$ is plotted in Fig. \ref{figheatper}.

The exact solution of this problem is given by:
\begin{equation} \label{heatpersol}
u(\x,t)=\frac{1}{4\pi t}\int_{\bR^2} e^{-\frac{|\x-\y|^2}{4t}} \tilde{f}(\y) d\y,
\end{equation}
where $\tilde{f}$ is the periodic extension of $f$.
This is precisely what is computed by the periodic version of the volume 
FGT and the solution $u(\x,t)$ is plotted for various choices
of $t$ in Fig. \ref{figheatper}, with nine digits of precision in the FGT.

\begin{remark}
There is a subtle issue regarding the use of the FGT to compute
\eqref{heatpersol}, namely that the error estimates for the FGT 
derived above are based
on the Gaussian rather than the heat kernel, which includes the 
additional $1/(4 \pi t)$ scaling in two dimensions.
To compute an accurate convolution requires 
that the local tables be built using the full heat kernel (whose support to 
a fixed precision $\epsilon$ is slightly greater than the support of 
the Gaussian alone). The far field and local expansions also require 
a few more terms. Without entering 
into a detailed analysis, we illustrate the difference
when $t = 10^{-4}$ for leaf node boxes in the present example.
For the FGT, Hermite expansions of order $p = 22$ are needed to 
achieve 9 digits of precision. For the full heat kernel,
it turns out that $p = 28$ is required to achieve the same accuracy.
\end{remark}

\begin{figure}[htbp] 
\centering
\includegraphics[width=.45\textwidth]{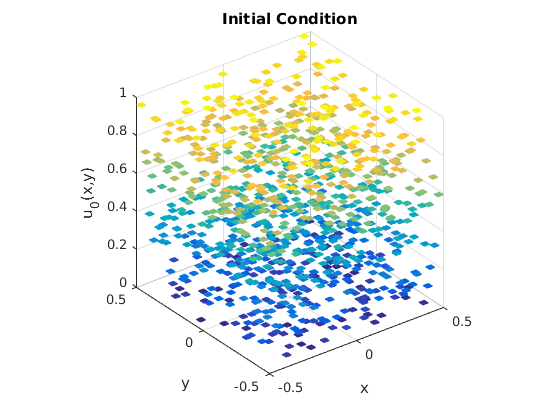} \hspace{.15in}
\includegraphics[width=.45\textwidth]{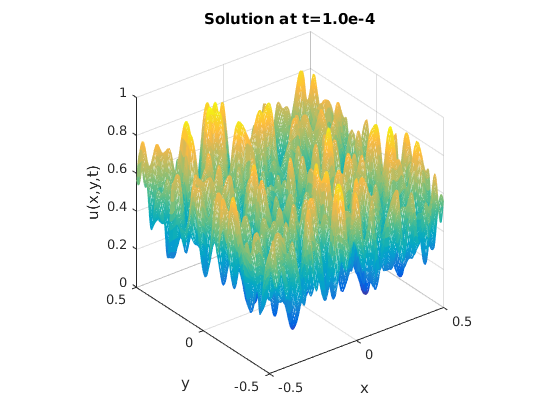}
\includegraphics[width=.45\textwidth]{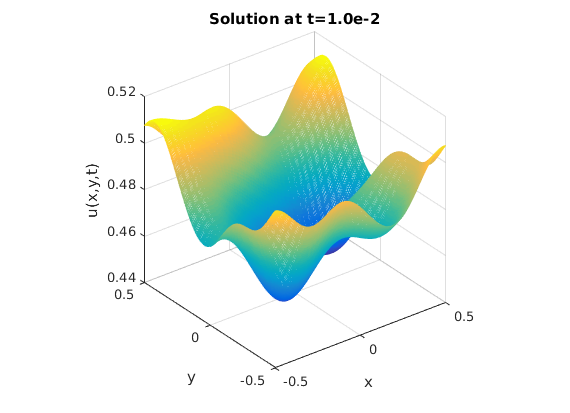} \hspace{.15in}
\includegraphics[width=.45\textwidth]{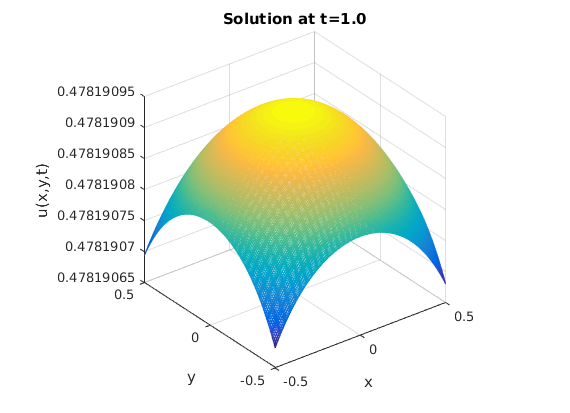}
\caption{The solution to the heat equation on a box with periodic 
boundary conditions and piecewise constant initial data, 
plotted at various times}
\label{figheatper}
\end{figure}

\section{Conclusions} \label{sec_conclusions}

We have presented a new adaptive version of the fast Gauss transform (FGT)
which can be used for the evaluation of volume or boundary integrals with 
a Gaussian kernel as well as the field induced by discrete point sources.
This is a standard and well-defined computational task in its own right, 
and serves as a key component in integral equation-based solvers for the
heat equation in complex geometry \cite{fullheatsolver}. 
The extension of the present method to 
three dimensions is straightforward and will be reported at a later date.

\bibliographystyle{siam}
\bibliography{ref.bib}

\begin{thebibliography}{10}

\bibitem{Arnold89}
{\sc D.~N. Arnold and P.~Noon}, {\em Coercivity of the single layer heat
  potential}, J. Comput. Math., 7 (1989), pp.~100--104.

\bibitem{AskhamCerfon}
{\sc T.~Askham and A.~J. Cerfon}, {\em An adaptive fast multipole accelerated
  poisson solver for complex geometries}, J. Comput. Phys., 344 (2017),
  pp.~1--22.

\bibitem{baxter}
{\sc B.~J.~C. Baxter and G.~Roussos}, {\em A new error estimate of the fast
  {G}auss transform}, SIAM Journal of Scientific Computing, 24 (2002),
  pp.~257--259.

\bibitem{boyd}
{\sc J.~P. Boyd}, {\em Chebyshev and Fourier Spectral Methods}, Dover, New
  York, 2001.

\bibitem{Brown89}
{\sc R.~M. Brown}, {\em The method of layer potentials for the heat equation in
  {L}ipschitz cylinders}, Amer. J. Math., 111 (1989), pp.~339--379.

\bibitem{Huang2006}
{\sc H.~Cheng, J.~Huang, and T.~J. Leiterman}, {\em An adaptive fast solver for
  the modified helmholtz equation in two dimensions}, Journal of Computational
  Physics, 211 (2006), pp.~616--637.

\bibitem{costabel}
{\sc M.~Costabel}, {\em Time-dependent problems with the boundary integral
  equation}, in Encyclopedia of Computational Mechanics, E.~Stein, R.~de~Borst,
  and T.~J.~R. Hughes, eds., John Wiley \& Sons, 2004, pp.~703--721.

\bibitem{dargushbanerjee}
{\sc G.~F. Dargush and P.~K. Banerjee}, {\em Application of the boundary
  element method to transient heat conduction}, International Journal for
  Numerical Methods in Engineering, 31 (1991), pp.~1231--1247.

\bibitem{DR}
{\sc P.~J. Davis and P.~Rabinowitz}, {\em Methods of numerical integration},
  Academic Press, San Diego, 1984.

\bibitem{elgammal03}
{\sc A.~Elgammal, R.~Duraiswami, and L.~S. Davis}, {\em Efficient kernel
  density estimation using the fast {G}auss transform with applications to
  color modeling and tracking}, IEEE Transactions on Pattern Analysis and
  Machine Intelligence, 25 (2003), pp.~1499--1504.

\bibitem{Ethridge2001}
{\sc F.~Ethridge and L.~Greengard}, {\em {A New Fast-Multipole Accelerated
  Poisson Solver in Two Dimensions}}, SIAM Journal on Scientific Computing, 23
  (2001), pp.~741--760.

\bibitem{greengard_lin}
{\sc L.~Greengard and P.~Lin}, {\em Spectral approximation of the free-space
  heat kernel}, Appl. Comput. Harmon. Anal., 9 (2000), pp.~83--97.

\bibitem{MLFMM}
{\sc L.~Greengard and V.~Rokhlin}, {\em A fast algorithm for particle
  simulations}, Journal of computational physics, 73 (1987), pp.~325--348.

\bibitem{greengard_strain1}
{\sc L.~Greengard and J.~Strain}, {\em The fast {G}auss transform}, SIAM J.
  Sci. Statist. Comput., 12 (1991), pp.~79--94.

\bibitem{greengard98}
{\sc L.~Greengard and X.~Sun}, {\em A new version of the fast {G}auss
  transform}, Documenta Mathematica, III (1998), pp.~575--584.

\bibitem{guentherlee}
{\sc R.~B. Guenther and J.~W. Lee}, {\em Partial differential equations of
  mathematical physics and integral equations}, Prentice Hall, Inglewood
  Cliffs, New Jersey, 1988.

\bibitem{powerheat}
{\sc M.~T. Ibanez and H.~Power}, {\em An efficient direct bem numerical scheme
  for phase change problems using fourier series}, Computer Methods in Applied
  Mechanics and Engineering, 191 (2002), pp.~2371--2402.

\bibitem{Langston2011}
{\sc H.~Langston, L.~Greengard, and D.~Zorin}, {\em A free-space adaptive
  fmm-based pde solver in three dimensions}, Comm. Appl. Math. and Comp. Sci.,
  6 (2011), pp.~79--122.

\bibitem{Lee2006}
{\sc D.~Lee, A.~Gray, and A.~Moore}, {\em Dual-tree fast gauss transforms},
  Advances in Neural Information Processing Systems, 18 (2006), pp.~747--754.

\bibitem{GLee98}
{\sc J.-Y. Lee and L.~Greengard}, {\em {A direct adaptive Poisson solver of
  arbitrary order accuracy}}, Journal of computational physics, 125 (1996),
  pp.~415--424.

\bibitem{GLiheatquad}
{\sc J.-R. Li and L.~Greengard}, {\em High order accurate methods for the
  evaluation of layer heat potentials}, SIAM J. Sci. Comput., 31 (2009),
  pp.~3847--3860.

\bibitem{Malhotra2016}
{\sc D.~Malhotra and G.~Biros}, {\em Algorithm 967: A distributed-memory fast
  multipole method for volume potentials}, ACM Trans. Math. Softw., 43 (2016),
  pp.~17:1--17:27.

\bibitem{pogorzelski}
{\sc W.~Pogorzelski}, {\em Integral equations and their applications}, Pergamon
  Press, Oxford, 1966.

\bibitem{pfgt}
{\sc R.~S. Sampath, H.~Sundar, and S.~Veerapaneni}, {\em Parallel fast gauss
  transform}, in SC '10: Proceedings of the ACM/IEEE International Conference
  for High Performance Computing, Networking, Storage and Analysis, New
  Orleans, LA, 2010, pp.~1--10.

\bibitem{fgt3}
{\sc M.~Spivak, S.~Veerapaneni, and L.~Greengard}, {\em The fast generalized
  gauss transform}, SIAM J. Sci. Comput., 32 (2010), pp.~3092--3107.

\bibitem{Strain1991}
{\sc J.~Strain}, {\em The fast gauss transform with variable scales}, SIAM J.
  Sci. Stat. Comput., 12 (1991), pp.~1131--1139.

\bibitem{strain_adapheat}
\leavevmode\vrule height 2pt depth -1.6pt width 23pt, {\em Fast adaptive
  methods for the free-space heat equation}, SIAM J. Sci. Comput., 15 (1994),
  pp.~185--206.

\bibitem{tausch-fgt}
{\sc J.~Tausch and A.~Weckiewicz}, {\em Multidimensional fast {G}auss
  transforms by {C}hebyshev expansions}, SIAM J. Sci. Comput., 31 (2009),
  pp.~3547--3565.

\bibitem{veerapanenibiros1}
{\sc S.~K. Veerapaneni and G.~Biros}, {\em A high-order solver for the heat
  equation in 1d domains with moving boundaries}, SIAM J. Sci. Comput., 29
  (2007), pp.~2581--2606.

\bibitem{veerapanenibiros2}
\leavevmode\vrule height 2pt depth -1.6pt width 23pt, {\em The {C}hebyshev fast
  {G}auss and nonuniform fast {F}ourier transforms and their application to the
  evaluation of distributed heat potentials}, J. Comput. Phys., 227 (2008),
  pp.~7768--7790.

\bibitem{wankarniadakis}
{\sc X.~Wan and G.~Karniadakis}, {\em A sharp error estimate for the fast
  {G}auss transform}, Journal of Computational Physics, 219 (2006), pp.~7--12.

\bibitem{fullheatsolver}
{\sc J.~Wang, L.~Greengard, S.~Jiang, and S.~K. Veerapaneni}, {\em A high-order
  solver for the two-dimensional heat equation in moving domains}, in
  preparation,  (2017).

\end{thebibliography}
\end{document}